\documentclass[twoside,11pt]{amsart}

\usepackage{enumerate}
\usepackage{graphicx}
\usepackage{amsmath}
\usepackage{amssymb}
\usepackage{epsfig}
\usepackage{amsthm}
\usepackage{amscd}
\usepackage{times}
\usepackage[latin2]{inputenc}
\def\r{\mathbb{R}}
\def\b{\mathbb{B}}
\def\d{\mathbb{D}}
\def\n{\mathbb{N}}
\def\c{\mathbb{C}}
\def\s{\mathbb{S}}

\def\pg{\mathfrak{p}}
\def\nor{\mathcal{N}}

\def\dc{\mathcal{D}}
\def\ve{{\varepsilon}}
\def\de{{\delta}}
\def\g{{\gamma}}

\def\a{{\alpha}}

\newcommand{\df}{ \stackrel{\rm def}{=}}

\newcommand{\rth}{\r^3}

\newcommand{\Int}{\operatorname{Int}}
\newcommand{\Div}{\operatorname{Div}}
\newcommand{\wt}{\widetilde}
\def\E{\textsc{e}}
\def\ri{{\rm i}}

\newcommand{\jj}{\operatorname{Jac}_{\lambda_1, \ldots,\lambda_{2(\sigma+k)}}}

\newtheorem{theorem}{Theorem}

\newtheorem{definition}{Definition}
\newtheorem{assertion}{Assertion}[section]
\newtheorem{lemma}{Lemma}
\newtheorem{corollary}{Corollary}
\newtheorem{proposition}{Proposition}
\newtheorem{remark}{Remark}
\newtheorem{conjecture}{Conjecture}
\newtheorem{example}{Example}

\newcommand{\dist}{\operatorname{dist}}
\newcommand{\re}{\operatorname{Re}}
\newcommand{\im}{\operatorname{Im}}

\newcommand{\intc}{\operatorname{Int}}

\newcommand{\cd}{\mathcal{D}}
\newcommand{\cm}{\mathcal{M}}
\newcommand{\cj}{\mathcal{J}}
\newcommand{\Z}{\mathbb Z}
\newcommand{\HH}{\mathbb{H}}
\newcommand{\ed}{\end{document}}

\newcommand{\ben}{\begin{enumerate}}
\newcommand{\een}{\end{enumerate}}

\begin{document}

\title[proper minimal surfaces of arbitrary
topological type]{existence of proper minimal surfaces of arbitrary
topological type}

\author[L. Ferrer]{Leonor Ferrer}
\author[F. Martin]{Francisco Martín}
\author[W. H. Meeks III]{William H. Meeks, III}

\thanks{This research is partially supported by MEC-FEDER Grant no. MTM2007 - 61775.
This material is based upon work for the NSF under Award No. DMS -
0703213. Any opinions, findings, and conclusions or recommendations
expressed in this publication are those of the authors and do not
necessarily reflect the views of the NSF}
\date{\today}

\address[L. Ferrer]{Departamento de Geometría y Topología. Universidad de Granada. 18071, Granada, Spain.}
\address[F. Martín]{Departamento de Geometría y Topología. Universidad de Granada. 18071, Granada, Spain.}
\address[W. H. Meeks III]{Mathematics Department at the University of Massachusetts.
Amherst, MA 01003, USA.}

\email[L. Ferrer]{lferrer@ugr.es}
\email[F. Martín]{fmartin@ugr.es}
\email[W. H. Meeks III]{bill@math.umass.edu}

\begin{abstract}  Consider a domain $\mathcal{ D} $ in $\rth$ which
is convex (possibly all $\rth$) or which is smooth and bounded.
Given any open surface $M$, we prove that there exists a complete,
proper minimal immersion $f \colon M \rightarrow \mathcal{D}$.
Moreover, if ${\mathcal D}$ is smooth and bounded, then we prove
that the immersion $f\colon M\to {\mathcal D}$ can be chosen so that
the limit sets of distinct ends of $M$ are disjoint connected
compact sets in $\partial {\mathcal D}$.

\vspace{.2cm} \noindent   {\it 2000 Mathematics Subject
Classification.} Primary 53A10; Secondary 49Q05, 49Q10, 53C42.
\end{abstract}
\maketitle

\section{Introduction}
A natural question in the global theory of minimal surfaces, first
raised by Calabi in 1965 \cite{calabi} and later revisited by Yau
\cite{Yau, Yau-2}, asks whether or not there exists a complete
immersed minimal surface in a bounded domain $\cd$ in $\rth$. In
1996, Nadirashvili \cite{na1} provided the first example of a
complete, bounded, immersed minimal surface in $\rth$. However,
Nadirashvili's techniques did not provide properness of such a
complete minimal immersion in any bounded domain. Under certain
restrictions on $\cd$ and the topology of an open
surface\footnote{We say that a surface is {\it open} if it is
connected, noncompact and without boundary.} $M$, Alarcón, Ferrer,
Martín, and Morales \cite{density, marmor1, marmor2, marmor3, propi}
proved the existence of a complete, proper minimal immersion of $M$
in $\cd.$

In this paper we prove that every open surface $M$ can be properly
minimally immersed into certain domains $\cd$ of $\r^3$ as a
complete surface (see Theorem~\ref{th:first}). These domains include
$\r^3$, all convex domains and all bounded domains with smooth
boundary. In contrast to this existence theorem,  Martín and Meeks \cite{mm}
have recently proven that in any Riemannian three-manifold there exist 
many nonsmooth domains with compact closure which do not admit any complete, properly 
immersed surfaces with at least one annular end and bounded mean 
curvature. The above result is a generalization of a previous work for 
minimal surfaces in $\r^3$ by these
authors and Nadirashvili \cite{mmn}. Thus, some geometric constraint on the boundary of a 
bounded domain is necessary to
insure that it contains complete, properly immersed minimal surfaces
of arbitrary topological type.

When the domain ${\cd}$ is smooth and bounded, we obtain further
important control on the limit sets of the ends of $M$ as described
in the next theorem; see Definition~\ref{def:limit} for the
definition of the limit set of an end.

\begin{theorem} \label{th:limit} If $\cd$ is a smooth bounded domain
in $\rth$ and $M$ is an open surface, then there exists a complete,
proper minimal immersion of $M$ in $\cd$ such that the limit sets of
distinct ends of $M$ are disjoint.
\end{theorem}

We consider the proof of the above theorem to be the first key point
in an approach by the second two authors and Nadirashvili to
construct certain complete, properly embedded minimal surfaces $M$
in certain bounded domains of $\rth$ as described in the next
conjecture. The cases described in this conjecture where $M$ is
nonorientable appear to be deeper and more interesting than where
$M$ is orientable. Our approaches for dealing with the orientable or
nonorientable cases in this conjecture are essentially the same by
using the theory developed in Section~\ref{sec:no}; specifically, we
refer the reader to Theorem~\ref{th:t0} and Propositions~\ref{th:no}
and \ref{prop:infinity}, which are closely related to parts 2 and 3
of the next conjecture.

\begin{conjecture}[Embedded Calabi-Yau Conjecture,\; Martín, Meeks,
Nadirashvili, Per\'{e}z, Ros] \label{conj:mmn} $\mbox{}$

\begin{enumerate}
\item A necessary and sufficient condition for an open surface
$M$ to admit complete, proper minimal embeddings in {\bf every}
smooth bounded domain in $\rth$ is that $M$ is orientable and every
end of $M$ has infinite genus.

\item A necessary and sufficient condition for an open surface
$M$ to admit a complete, proper minimal embedding in {\bf some}
smooth bounded domain in $\rth$ is that every end of $M$ has
infinite genus and $M$ has only a finite number of nonorientable
ends.

\item Let $\cd_\infty$ be the bounded domain in $\rth$ described
in Example~\ref{ex:universal}, which is smooth except at one point
(see Fig.~\ref{fig:fig-7}). A necessary and sufficient condition for
an open surface $M$ to admit a complete, proper minimal embedding in
$\cd_\infty$ is that every end of $M$ has infinite genus.
\end{enumerate}
\end{conjecture}

Embeddedness creates a dichotomy in the Calabi-Yau question. In
other words, when the question is asked whether a given domain of
$\rth$ admits a complete,  injective minimal immersion of a  surface
$M$, the topological possibilities are limited. The first result
concerning the embedded Calabi-Yau question was given  by Colding
and Minicozzi \cite{cm35}. They proved that complete, embedded
minimal surfaces in $\r^3$ with finite topology are proper in
$\r^3$. The relevance of their result to the classical theory of
complete embedded minimal surfaces is that there are many deep
theorems concerning properly embedded minimal surfaces.  Recently,
Meeks, Pérez and Ros \cite{mpr-2} generalized this properness result
of Colding and  Minicozzi to the larger class of surfaces with
finite genus and a countable number of ends.

There are many known topological obstructions for properly minimally
embedding certain open surfaces into $\r^3$. For example, the only
properly embedded, minimal planar domains in $\r^3$ are the plane
and the helicoid which are simply-connected, the catenoid which is
1-connected and the Riemann minimal examples which are planar
domains with two limits ends (see \cite{lopez-ros, collin, mr8,
mpr-1, mpr-3} for this classification result). Because of these
results, the proper minimal  immersions described in this paper must
fail to be embeddings for certain open surfaces.

The constructive nature of the proper minimal surfaces in our
theorems depends on the bridge principle for minimal surfaces and on
generalizing to the nonorientable setting the approximation
techniques used by Alarcón, Ferrer and Martín in \cite{density}.
Also, the construction of the surfaces which we obtain here depend
on obtaining certain compact exhaustions for any open surface $M$;
see Section~\ref{sec:simple} for the case of orientable open
surfaces and the proofs of Propositions~\ref{th:no} and
\ref{prop:infinity} in Section~\ref{sec:6.3} for the case of
nonorientable open surfaces. \vskip 3mm

\noindent \textbf{Acknowledgments.} We are indebted to Nikolai
Nadirashvili for sharing with us his invaluable insights into
several aspects of this theory. We would like to thank Joaquin Pérez
for making some of the figures in this paper and Francisco J. L\'{o}pez
for helpful discussions on the material in Section~\ref{subsec:non}.

\section{Preliminaries and Background}

\subsection{Background on convex bodies and Hausdorff distance}

Given $E$ a bounded regular convex domain of $\r^3$ and $p \in
\partial E$, we will let $\kappa_2(p) \geq \kappa_1(p) \geq 0$ denote the
principal curvatures of $\partial E$ at $p$ (associated to the
inward pointing unit normal). Moreover, we write:
$$\kappa_1(\partial E) \df \mbox{min} \{ \kappa_1(p) \mid p \in
\partial E \} , \qquad \kappa_2(\partial E) \df \mbox{max} \{
\kappa_2(p) \mid p \in \partial E \}.$$ If  we consider $\nor\colon
\partial E \rightarrow \s^2$ to be the outward pointing unit normal or
Gauss map of $\partial E$, then there exists a constant $a>0$
(depending on $E$) such that $\partial E_t=\{p+ t\cdot \nor(p) \mid
p \in \partial E\}$ is a regular (convex) surface for all $ t \in
[-a, +\infty[$. Let  $E_t$ denote the convex domain bounded by
$\partial E_t$. The normal projection to $\partial E$ is represented
as
$$\begin{array}{rccl} \mathcal{P}_E: &\r^3- E_{-a}& \longrightarrow &
\partial E \\ \quad & p+t \cdot \mathcal{N}(p) & \longmapsto
& p \; .\end{array}$$

For a subset $\Upsilon$ in $\r^3$ and a  real $r>0$, we define the
tubular neighborhood of radius $r$ along $\Upsilon$ in the following
way: $T(\Upsilon,r) = \Upsilon + \b(0,r),$ where $\b(0,r)=\{ p \in
\r^3 \mid \|p\|<r\}.$

A convex set of $\r^n$ with nonempty interior is called {\em a
convex body}. The set $\mathcal{ C}^n$ of convex bodies  of $\r^n$
can be made into a metric space in several geometrically reasonable
ways. The Hausdorff metric is particularly convenient and applicable
for defining such a metric space structure. The natural domain for
this metric is the set $\mathcal{ K}^n$ of the nonempty compact
subsets of $\r^n$. For $\mathcal{C}$, $\mathcal{D}\in \mathcal{
K}^n$ the {\em Hausdorff distance} is defined by:

$$ \delta^H(\mathcal{C},\mathcal{D})=\min \left\{ \lambda \geq 0 \; | \; \mathcal{C}
\subset T(\mathcal{D},\lambda), \; \mathcal{D} \subset
T(\mathcal{C},\lambda) \right\}.$$  A theorem of H.
Minkowski (cf. \cite{mingorebulgo}) states that every convex body
$C$ in $\r^n$ can be approximated (in terms of Hausdorff metric) by
a sequence $C_k$ of `analytic' convex bodies.
\begin{theorem}[Minkowski] \label{th:minko}
Let $\mathcal{C}$ be a convex body in $\r^n$. Then there exists a
sequence $\{\mathcal{C}_k \}$ of convex bodies with the following
properties
\begin{enumerate}[\rm 1.]
\item $\mathcal{C}_k \searrow \mathcal{C}$;
\item $\partial \, \mathcal{C}_k$ is an analytic $(n-1)$-dimensional manifold;
\item The principal curvatures of $\partial \, \mathcal{C}_k$ never vanish.
\end{enumerate}
\end{theorem}
A modern proof of this result can be found in \cite[\S 3]{meeksyau}.


\subsection{Preliminaries on minimal surfaces}
\label{subsec:minimal}

Throughout the paper, whenever we write that $M$ is a {\em compact
minimal surface with boundary}, we will mean that this boundary is
regular and $M$ can be extended beyond its boundary. In other words,
we will always assume that $M \subset \Int(M')$, where $M'$ is
another minimal surface.

For the sake of simplicity of notation and language, we will say
that two immersed surfaces in $\rth$ are {\it homeomorphic} if and
only if their underlying topological surface structures are the same.

The following lemma will be a key point (together with the bridge
principle and the existence of simple exhaustions) in the proofs of
the main lemmas of this paper. It summarizes all the information
contained in Lemma 5, Theorem 3 and Corollary 1 in \cite{density}.
\begin{lemma}[Alarcón, Ferrer, Martín] \label{lem:afm}
Let $\mathcal{D}'$ be a convex domain (not necessarily bounded or
smooth) in $\r^3$. Consider  a compact orientable minimal surface
$M$, with nonempty boundary satisfying: $\partial M \subset
\mathcal{D} - \overline{\mathcal{D}}_{-d}$, where $\mathcal{D}$ is a
bounded convex smooth domain, with $\overline{\mathcal{D}} \subset
\mathcal{D}'$, and $d>0$ is a constant. Let $r$ be a positive
constant such that $T(M,r) \subset \dc.$

Then, for any $\varepsilon>0$, there exists a complete minimal
surface $M_\varepsilon$ which is properly immersed in $\mathcal{D}'$
and satisfies:
\begin{enumerate}
\item $M_\varepsilon$ has  the same topological type as $\Int(M)$;
\item $M_\varepsilon \cap T(M, r)$ contains a connected surface $M_\varepsilon^r$
(not a component of $M_\varepsilon \cap T(M, r)$)
with the same topological type as $\Int(M)$ and
$M_\varepsilon^r$ converges in the $\mathcal{C}^\infty$ topology to
$M$, as $\varepsilon \to 0$. Furthermore, the Hausdorff distance
$\delta^H(M_\varepsilon^r,M)<\varepsilon$ ;
\item Each end of $M_\varepsilon-M_\varepsilon^r$ is contained
in $\r^3 - \mathcal{D}_{-2 d-\varepsilon}$;
\item If $\mathcal{D}$ and $\mathcal{D}'$ are smooth and $\mathcal{D}$ is strictly
convex, then $\delta^H(M,M_\varepsilon)< \mathsf{m}(\varepsilon,d,\mathcal{D},\mathcal{D}')$, where:
$$\mathsf{m}(\varepsilon,d,\mathcal{D},\mathcal{D}')\df \varepsilon+
\sqrt{\frac{2 (\delta^H(\mathcal{D},\mathcal{D}')+d+\varepsilon)}{\kappa_1(\partial \mathcal{D})}
+\left( \delta^H(\mathcal{D},\mathcal{D}')+d+\varepsilon \right)^2}.$$
\end{enumerate}
\end{lemma}

\subsubsection{The bridge principle for minimal surfaces}
Let $M$ be a possibly disconnected, compact minimal surface in
$\rth$, and let $P\subset \rth$ be a thin curved rectangle whose two
short sides lie along $\partial M$ and that is otherwise disjoint
from $M$. The {\em bridge principle} for minimal surfaces states
that  if $M$ is nondegenerate, then it should be possible to deform
$ M \cup P$ slightly to make a minimal surface with boundary
$\partial (M \cup P)$. The bridge principle is a classical problem
that goes back to Paul Lévy in the 1950's. It was involved in the
construction of a curve bounding uncountably many minimal disks. The
bridge principle is easy to apply to compact minimal surfaces which
satisfy the nondegerancy property described in the next definition.

\begin{definition}
A compact minimal surface $M$ with boundary is said to be {\bf
nondegenerate} if there are no nonzero Jacobi fields on $M$ which
vanish on $\partial M$.
\end{definition}

The following version of the bridge principle is the one we need in
our constructions.
\begin{theorem}[White, \cite{W1,W2}] \label{th:bridge}
Let $M$ be a compact, smooth, nondegenerate minimal surface with
boundary, and let $\Gamma$ be a smooth arc such that $\Gamma \cap
M=\Gamma \cap \partial M=\partial \Gamma.$

Let $P_n$ be a sequence of bridges on $\partial M$ that shrink
nicely to $\Gamma.$

Then for sufficiently large $n$, there exists a minimal surface
$M_n$ with boundary $\partial (M \cup P_n)$ and a diffeomorphism
$f_n \colon M \cup P_n \to M_n$ such that
\begin{enumerate}[ \rm \; \; (1)]
\item {\rm area}$(M_n)$ $\to $ {\rm area}$(M)$;
\item $f_n(x)\equiv x$ for all $x \in \partial (M \cup P_n)$;
\item $\|x-f_n(x)\|=O(w_n)$, where $w_n$ is the width of $P_n$
and $O(w_n)/w_n$ is bounded;
\item The maps ${f_n}|_{M}$ converge smoothly on compact subsets of
$M-\Gamma$ to the
identity map $1_M \colon M \to M;$
\item Each $M_n$ is a nondegenerate minimal surface.
\end{enumerate}
\end{theorem}
\section{Adding handles and ends}
In this section we prove two lemmas which represent main tools in
our construction procedure. Essentially, they tell to us how we can
add a ``pair of  pants'' to a minimal surface with boundary in order
to create a new end (Figure~\ref{fig:Drawing-1-2}.(a)) or how to add
a handle to increase the genus (Figure~\ref{fig:Drawing-1-2}.(b)).
\begin{figure}[htbp]
    \begin{center}
      \includegraphics[width=\textwidth]{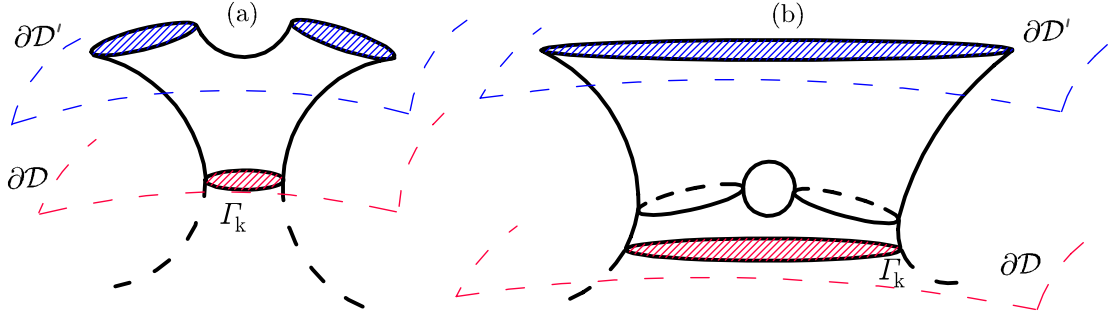}
   \end{center}
   \caption{We can add a ``pair of  pants'' to a minimal surface with boundary
   in order to: (a) create a new end,  or (b) add a handle.}\label{fig:Drawing-1-2}
\end{figure}


\begin{lemma}[{\bf Adding ends}] \label{lem:adding-ends}
Let $\cd$ and $\cd'$ be two smooth bounded strictly convex domains
in $\r^3$ so that $\vec 0 \in \overline{\cd} \subset \cd'$. Consider
a compact minimal surface $M$ with nonempty boundary and satisfying
$\vec 0 \in \Int(M)$ and $\partial M \subset \partial \cd$. Assume
that $M$ has genus $g$ and $k$ components at the boundary ($k \geq
1$), $\partial M = \Gamma_1 \cup \ldots \cup \Gamma_k$. We also
assume that $M$ intersects $\partial \cd$ transversally.

Then for any $\varepsilon>0$, there exists a minimal surface $M_\varepsilon$
satisfying the following properties:
\begin{enumerate}
\item $M_\varepsilon$ is a smooth, immersed minimal surface with genus
$g$ and $k+1$ boundary components. Moreover, $\partial
M_\varepsilon \subset \partial \cd'$,  $\partial
M_\varepsilon $ meets transversally $ \partial \cd'$and $\vec 0 \in \Int
(M_\varepsilon);$
\item The intrinsic distance
$\dist_{M_\varepsilon} (\vec 0, \partial M_\varepsilon) > \dist_M(\vec 0,\partial M)+1;$
\item The surfaces $M_\varepsilon \cap \overline{\dc}$ are graphs over $M$ and
converge in the $C^\infty$ topology to $M$, as $\varepsilon \to 0$. Furthermore,
$\delta^H(M,M_\varepsilon \cap \overline{\dc})<\varepsilon;$
\item $M_\varepsilon - \dc$ consists of $k-1$ annuli, each of whose boundary in $\partial \cd$
lies in $T(\Gamma_j,\varepsilon)$, $j=1, \ldots,k-1$, and a pair of
pants, whose boundary in $\partial \cd$ is a single curve which lies
in  $T(\Gamma_k,\varepsilon)$  (see Figure
~\ref{fig:Drawing-1-2}-(a)). Moreover, the two boundary curves of
the pair of pants which are contained in $\partial \cd'$ are
disjoint;

\item If $\cd$ and $\cd'$ are parallel (boundaries are equidistant), then
$\delta^H(M,M_\varepsilon)< 2 \, C(\varepsilon,\mathcal{D},\mathcal{D}')$, where:
$$ C(\varepsilon,\mathcal{D},\mathcal{D}')\df
\varepsilon+\sqrt{\frac{2 (\delta^H(\mathcal{D},\mathcal{D}')+
2 \varepsilon)}{\kappa_1(\partial \mathcal{D}')}+\left( \delta^H(\mathcal{D},\mathcal{D}')
+2 \varepsilon \right)^2};$$


\end{enumerate}
\end{lemma}

\begin{proof} Fix $\varepsilon>0$.
The proof of this lemma consists of clever combined applications of
the density theorem (Lemma~\ref{lem:afm}) and the bridge principle
(Theorem~\ref{th:bridge}). We have divided the proof into three
steps. \vskip .3cm

\noindent {\bf Step 1}. From our assumptions, we  know that
 $M \subset \Int (M')$, where $M'$
is a regular minimal surface. Take $a>0$ small enough such  that
$\overline{\mathcal{D}}_a \subset \mathcal{D}',$ and $ \delta^H(M,M'
\cap \overline{\mathcal{D}}_a)< \varepsilon/4. $ Consider $d>0$ and
$\varepsilon_0>0$ such that $a>2 \,d +\varepsilon_0$ and
$\varepsilon>d+\varepsilon_0.$ Let $M''\subset \Int(M') \cap
\overline{\mathcal{D}}_a$ be a compact minimal surface with boundary
such that $M''$ is homeomorphic to $M$, $\partial M''\subset
\dc_a-\overline{\dc}_{a-d}$ and
\begin{equation} \label{eq:catarsis-1}
\delta^H(M,M'')< \varepsilon/4.
\end{equation}
Finally, take $r>0$ such that $T(M'',r)\subset \dc_a$. Given
$\varepsilon'' \in (0,\min\{\varepsilon_0,\varepsilon/4\}]$, then we
apply Lemma~\ref{lem:afm} to the data: $\varepsilon'',$ $d, $ $M'',$
$\mathcal{D}_a,$  and $\mathcal{D}'.$ So, we obtain a complete,
minimal surface $\widetilde M$ properly immersed in $\mathcal{D}'$,
which satisfies:

\begin{itemize}
\item $\widetilde M$ has the same topological type as $\Int(M'') \equiv \Int(M)$
and $0 \in \Int(\widetilde M)$;
\item The surface $\widetilde M \cap T(M'',r)$ contains a regular compact surface
$\widetilde M^r$ which is homeomorphic to $M''$ and these surfaces
converge smoothly to $M''$, as $\varepsilon'' \to 0$. Furthermore,
$\delta^H(\widetilde M^r,M'')< \varepsilon'';$
\item Each end of $\widetilde M - \widetilde M^r$ is contained
in $\mathcal{D}'- \overline{\mathcal{D}}$ (here, we use $a-2 \, d-\varepsilon_0>0$);
\item $\delta^H(\widetilde M, M'') < \varepsilon''+\sqrt{2 \,
\frac{\delta^H(\cd, \cd')+d-a+\varepsilon''}{\kappa_1(\partial \cd_a)}+
(\delta^H(\cd,\cd')+d-a+\varepsilon'')^2  }.$
\end{itemize}

Assume now that $\cd$ and $\cd'$ are parallel. From our assumptions
about $d$ and $\varepsilon''$ and taking into account that
$\kappa_1(\cd_a) \geq \kappa_1(\cd'),$ then the last inequality
becomes:
$$ \delta^H(\widetilde M, M'') < \frac{\varepsilon}{4}+\sqrt{2 \,
\frac{\delta^H(\cd, \cd')+\varepsilon}{\kappa_1(\partial \cd ')}
+(\delta^H(\cd,\cd')+\varepsilon)^2  }.$$

\noindent {\bf Step 2}. Consider now $a'>0$ such that
$\overline{\mathcal{D}}_a \subset \mathcal{D}'_{-2a'}$. Let
$\widetilde M'$ be a compact region of $\widetilde M$, with regular
boundary, and such that:
\begin{enumerate}[ \; \; ({\bf A}.1)]
\item $\partial \widetilde M' \subset \mathcal{D}'-\overline{\mathcal{D}'}_{-a'};$
\item  $ \widetilde M^r \subset \widetilde M' \subset \widetilde M$;
\item The origin $\vec 0 \in \Int(\widetilde M')$
and $\dist_{\widetilde M'}(\vec 0,\partial \widetilde M')>\dist_{M}(\vec 0,\partial M)+1;$
\item $ \displaystyle \delta^H(\widetilde M ', M'') < \frac{\varepsilon}{4}+\sqrt{2 \,
\frac{\delta^H(\cd, \cd')+\varepsilon}{\kappa_1(\partial \cd ')}+(\delta^H(\cd,\cd')+\varepsilon)^2  }.$
\end{enumerate}
Take $\varepsilon'_0\in (0,\frac
\varepsilon 4)$ such that $\overline{\mathcal{D}}_a \subset
\mathcal{D}'_{-2a'-\varepsilon'_0}$.

\begin{figure}[htbp]
    \begin{center}
        \includegraphics[width=0.75\textwidth]{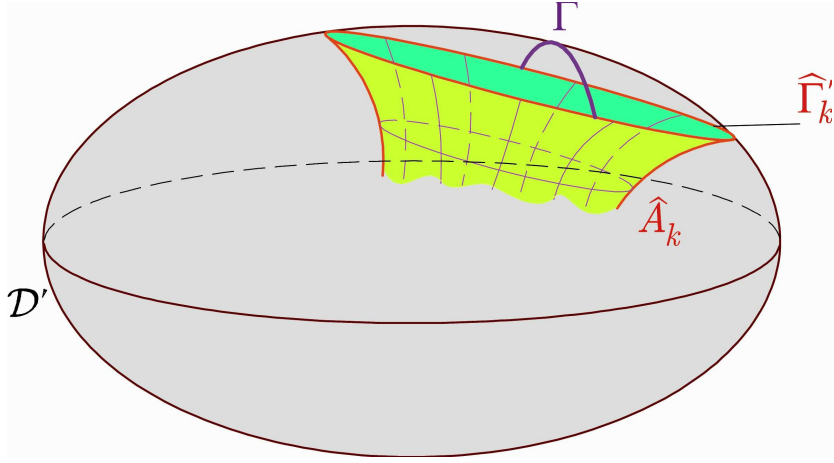}
    \end{center}
    \caption{Let  $\widehat A_k \subset \widehat M'$ be the (closed) annulus
    bounded by $\widehat \Gamma_k$ in $\overline{\cd'}-\cd$.
    Label by $\widehat \Gamma_k'$ the boundary
component of $\widehat A_k$ in $\partial \cd'$. Then, we connect two points $p$
and $q$ in $\widehat \Gamma_k'$ by a
simple smooth arc, $\Gamma \subset \partial \cd'$ so that the bridge
principle can be applied to the configuration $\widehat M' \cup \Gamma$.}
    \label{fig:ends-1}
\end{figure}

At this point, we apply again  Lemma~\ref{lem:afm} to the convex
domains $\mathcal{D}'_{b}$, $\mathcal{D}'$, the constants $d=a'$,
$\varepsilon'\in (0,\varepsilon'_0]$,  $r'>0$, and the compact
minimal surface $\widetilde M'$. Thus, we obtain a complete minimal
surface $\widehat M$ which is properly immersed in $\mathcal{D}'_b$
and satisfies the following conditions:

\begin{itemize}
\item $\widehat M$ has the same topological type as $\Int(\widetilde M')$
(which is homeomorphic to $\Int(M)$), and $\vec 0 \in \Int(\widehat M)$;
\item The surface $\widehat M \cap T(\widetilde M',r')$ contains a regular compact surface
$\widehat M^{r'}$ which is homeomorphic to $\widetilde M'$ and these
surfaces converge smoothly to $\widetilde M'$, as $\varepsilon'\to
0$. Furthermore, $\delta^H(\widehat M^{r'},M'')< \varepsilon';$
\item Each end of $\widehat M - \widehat M^{r'}$ is contained
in $\mathcal{D}'_b- \overline{\mathcal{D}_a}$ (here, we use $\overline{\mathcal{D}}_a \subset
\mathcal{D}'_{-2a'-\varepsilon'_0}$);
\item $\delta^H(\widetilde M', \widehat M)<\varepsilon'+\sqrt{2 \frac{b+a'
+\varepsilon'}{\kappa_1(\partial \cd')}+(b+a'+\varepsilon')^2}.$
\end{itemize}
Notice that if $b$, $a'$ and
$\varepsilon'$ are taken small enough in terms of $\kappa_1(\cd')$,
then the last inequality becomes:
\begin{equation} \label{eq:catarsis-2}
\delta^H \left(\widetilde M', \widehat M \cap \overline{\cd'} \right)<\varepsilon/4.
\end{equation}

\noindent {\bf Step 3}. Finally, we consider $\widehat M' $ a
connected component of $ \widehat M\cap \overline{\cd'}$ with the
same topological type as $M$. Up to an infinitesimal homothety, we
can assume that $\widehat M'$ meets $\partial \cd'$ transversally
and that $\widehat M'$ is {\em nondegenerate}. Let $\widehat
\Gamma_k$ denote the component of $\widehat M' \cap  \partial \cd$
which is contained in the tube $T(\Gamma_k, \frac{\varepsilon}{2})$
and let  $\widehat A_k \subset \widehat M'$ be the (closed) annulus
bounded by $\widehat \Gamma_k$ in $\overline{\cd'}-\cd$. Label by
$\widehat \Gamma_k'$ the boundary component of $\widehat A_k$ in
$\partial \cd'$. Now, we connect two points $p$ and $q$ in $\widehat
\Gamma_k'$ by a simple smooth arc, $\Gamma \subset \partial \cd'$,
such that:
\begin{itemize}
\item $\Gamma \cap \widehat M'= \Gamma \cap \left( \partial \widehat M' \right)
= \partial \Gamma$, see Figure~\ref{fig:ends-1}.
\item $\delta^H\left(\Gamma \cup \widehat M', \widehat M' \right)< \varepsilon/4.$
\end{itemize}
Then we attach a thin bridge $B_1$ along the arc $\Gamma$ to the
surface $\widehat M'$ (see Figure~\ref{fig:ends-2}). This new
minimal surface is called $M_\varepsilon$. Notice that
$M_\varepsilon$ is nondegenerate (Theorem~\ref{th:bridge}) and, if
the bridge $B_1$ is thin enough, we also have:

\begin{equation} \label{eq:catarsis-3} \delta^H( M_\varepsilon,
\widehat M' )< \varepsilon/4. \end{equation}

\begin{figure}[!h]
    \begin{center}
        \includegraphics[width=0.7\textwidth]{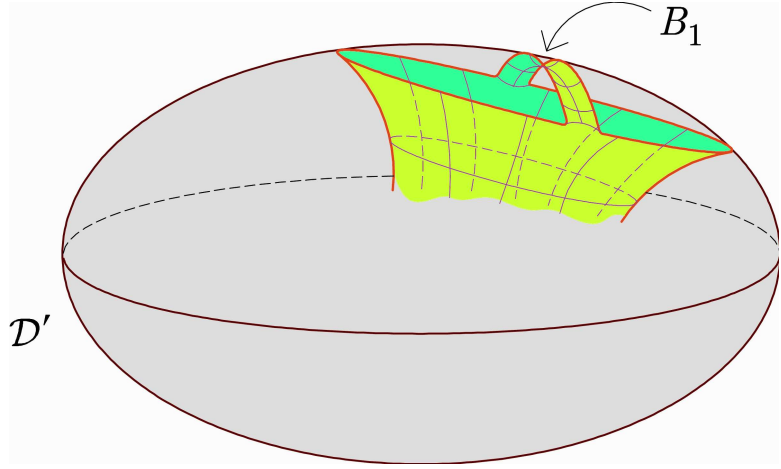}
    \end{center}
    \caption{We attach a thin bridge $B_1$ along the arc $\Gamma$
to the surface $\widehat M'$. In this way we obtain the desired surface $M_\ve.$}
    \label{fig:ends-2}
\end{figure}

Moreover, up to an infinitesimal translation and an infinitesimal
expansive dilation, we can assume that $\vec 0 \in \Int
(M_\varepsilon)$ and that $M_\ve$ can be extended beyond its
boundary. Taking into account that, outside an open neighborhood of
$\Gamma$,  $M_\ve$ converges smoothly to $\widehat M'$ as $\ve\to 0$
(Theorem~\ref{th:bridge}, item (4)), and the previously described
properties satisfied by $\widehat M'$ and $\widetilde M'$, then it
is not hard to see that $M_\ve$ satisfies items (1) to (4) in the
lemma. Item (5) is a direct consequence of the triangle inequality
and the inequalities \eqref{eq:catarsis-1}, (A.4),
\eqref{eq:catarsis-2}, and \eqref{eq:catarsis-3}.
\end{proof}


\begin{lemma}[Adding handles] \label{lem:adding-handles}
Let $\cd$ and $\cd'$ be two smooth bounded strictly convex domains
in $\r^3$ so that $\vec 0 \in \cd \subset \overline{\cd} \subset
\cd'$. Consider a compact minimal surface $M$ with nonempty boundary
and satisfying $\vec 0 \in \Int(M)$ and $\partial M \subset \partial
\cd$. Assume that $M$ has genus $g$ and $k$ boundary components ($k
\geq 1$), $\partial M = \Gamma_1 \cup \ldots \cup \Gamma_k$. We also
assume that $M$ intersects $\partial \cd$ transversally.

Then for any $\varepsilon>0$, there exists a minimal surface
$M_\varepsilon$ satisfying the following properties:
\begin{enumerate}
\item $M_\varepsilon$ is a smooth, immersed minimal surface with genus
$g+1$ and $k$  boundary components. Moreover, $\partial
M_\varepsilon \subset \partial \cd'$,
$\partial M_\varepsilon$ meets transversally $\partial \cd'$ and
$\vec 0 \in \Int (M_\varepsilon);$
\item The intrinsic distance $\dist_{M_\varepsilon} (\vec 0, \partial M_\varepsilon)
> \dist_M(\vec 0,\partial M)+1;$
\item The surfaces $M_\varepsilon \cap \overline{\dc}$ are graphs over
$M$ and converge in the $C^\infty$ topology to $M$, as $\varepsilon \to 0$. Furthermore,
$\delta^H(M,M_\varepsilon \cap \overline{\dc})<\varepsilon;$
\item $M_\varepsilon - \dc$ consists of $k-1$ annuli, whose boundary in $\partial \cd$
lies in $T(\Gamma_j,\varepsilon)$, $j=1, \ldots,k-1$, and an annulus
with a handle,  whose boundary in $\partial \cd$ is a single curve
which lie in  $T(\Gamma_k,\varepsilon)$  (see Figure
~\ref{fig:Drawing-1-2}-(b));

\item If $\cd$ and $\cd'$ are parallel, then $\delta^H(M,M_\varepsilon)< 2 \, C(\varepsilon,\mathcal{D},\mathcal{D}')$,
where the constant $C(\varepsilon,\mathcal{D},\mathcal{D}')$ is
given in Lemma~\ref{lem:adding-ends}.
\end{enumerate}

\end{lemma}

\begin{proof}
The proof of this lemma is identical to the one of Lemma
~\ref{lem:adding-ends}, except for Step 3 which is slightly
different. We construct the surface $M_\ve$, like in the third step
of the previous lemma. But this time we add a second bridge $B_2$
along a curve $\gamma$ joining two opposite points in $\partial B_1$
(see Figure~\ref{fig:ends-3}). Notice that, in this way, the old
annular component $\widehat A_k$ becomes an annulus with a handle.
\begin{figure}[!h]
    \begin{center}
        \includegraphics[width=0.7\textwidth]{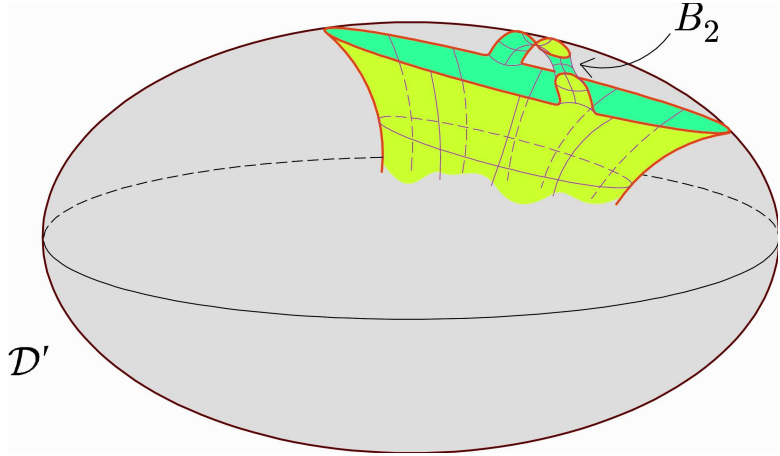}
    \end{center}
    \caption{This time, we construct the surface $M_\ve$, like in the third
    step of Lemma~\ref{lem:adding-ends}. But this time we add a second bridge $B_2$
along a curve $\gamma$ joining two opposite points in $\partial B_1$}
    \label{fig:ends-3}
\end{figure}

\end{proof}

\section{The existence of simple exhaustions} \label{sec:simple}

In this section we  prove that any open orientable  surface $M$ of
infinite topology  has a smooth compact exhaustion $M_1\subset
M_2\subset \cdots M_n\subset \cdots$, called a \emph{simple
exhaustion}. The defining properties for this exhaustion to be
simple when $M$ is orientable are:
 \begin{enumerate}[\bf \; \; \; 1.]
\item $M_1$ be a disk.
\begin{center} For all $n \in \n:$ \end{center}
\item Each component of $M_{n+1}-\Int(M_n)$ has one boundary component in
$\partial M_n$ and at least one boundary component in $\partial M_{n+1}$.
\item $M_{n+1}-\Int(M_n)$ contains a unique nonannular component which
topologically is a pair of pants or an annulus with a handle.
\end{enumerate}
If $M$ has finite topology with genus $g$ and $k$ ends, then we call
the compact exhaustion {\em simple} if properties 1 and 2 hold,
property 3 holds for $n\leq g+k$, and when $n> g+k$, all of the
components of $M_{n+1}-\Int(M_n)$ are annular.

The reader should note that for any simple exhaustion of $M$, each
component of $M-\Int(M_n)$ is a smooth, noncompact proper subdomain
of $M$ bounded by a simple closed curve and for each $n \in \n$,
$M_n$ is connected (see Fig.~\ref{fig:simplex}).
\begin{figure}[htbp]
   \begin{center}
       \includegraphics[width=0.75\textwidth]{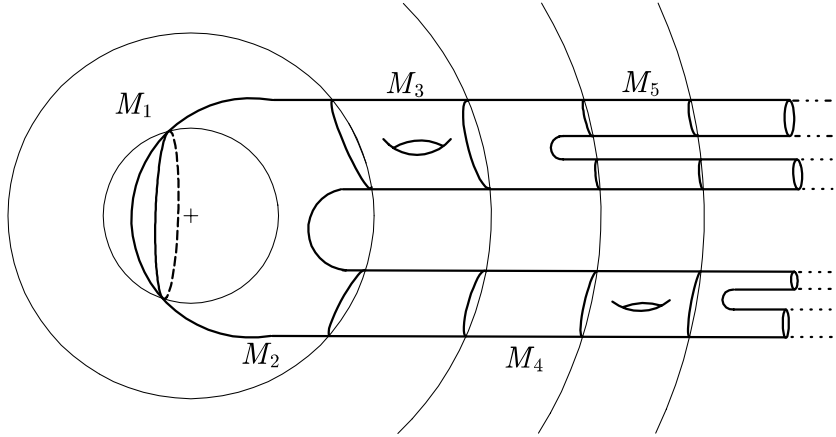}
   \end{center}
   \caption{A topological representation of the terms $M_1$ to $M_5$  in
   the exhaustion of the open surface $M$ given in Lemma~\ref{lem:simple}.}
   \label{fig:simplex}
\end{figure}

The following elementary lemma plays an essential role in the proofs
of Theorems 1 and 2.

\begin{lemma} \label{lem:simple} Every orientable open surface admits a simple exhaustion.
\end{lemma}
\begin{proof} If $M$ has finite topology, the proof of the
existence of a simple exhaustion is a straightforward consequence of
the arguments we are going to use in the infinite topology
situation. Assume now that $M$ has infinite topology.

Consider a smooth compact exhaustion $W_1 \subset \cdots \subset W_n
\subset \cdots$ of $M$ such that $W_1$ is a disk. We first show
that:

\begin{assertion} \label{as:two}
The exhaustion can be modified so that for every $j \in \n$,  $W_j$ is connected.
\end{assertion}
If this assertion fails to hold for the given exhaustion, there
exists a smallest $n>1$ such that $W_n$ consists of  a  finite
collection of components $W_n(1),\ldots,W_n(m)$ with  $m>1$ and
where $W_1 \subset W_n(1)$.  For each $j  \in \{2, \ldots, m\}$,
choose a smooth embedded arc $\alpha_j \subset M- \Int(W_n)$ joining
a point in the boundary of $W_n(j)$ to a point in the boundary of
$W_n(1)$ and so that these arcs form a pairwise disjoint collection.
Let $W_n'$ be the union of $W_n$ together with a closed regular
neighborhood in $M$ of the union of these arcs; $W_n'$ is connected
since $W_{n-1}$ is connected. Suppose $W_n' \subset W_{n+k}$ for
some $k$. Consider the new exhaustion $W_1 \subset \cdots \subset
W_{n-1}\subset W_n' \subset W_{n+k} \subset \cdots$ for $M$.
Repeating this argument inductively, one obtains  a  new compact
exhaustion  satisfying the connectedness  condition stated in the
assertion. \vskip 4mm

Assume now that the exhaustion fulfills the above assertion.
\begin{assertion} \label{as:three}
The exhaustion can be modified so that for all $j \in \n$, $W_j$ is connected and there are no compact
components in $M-\Int(W_j)$.
\end{assertion}
If assertion were to fail, then for some smallest $n$, $M-\Int(W_n)$
contains a maximal (possibly disconnected) compact domain $F$.  For
some $k>0$,  the connected compact domain $W_n \cup F$ is a subset of $W_{n+k}$ and so, we obtain
a new exhaustion
$$ W_1 \subset \cdots \subset W_{n-1} \subset W_n \cup F \subset W_{n+k} \subset \cdots.$$
Repeating this argument inductively, we obtain a new compact
exhaustion satisfying  the conclusions of Assertion~\ref{as:three}.
\vskip 4mm

Assume now that the exhaustion satisfies Assertion~\ref{as:three}.

\begin{assertion} \label{as:one} The exhaustion can be modified
so that, for every $j\in\n$, each boundary curve of $W_j$  separates $M$, each $W_j$
is connected and there are no compact components in $M - \Int(W_j)$.
\end{assertion}
If this new condition fails to hold for our given exhaustion, there
exists a smallest $n>1$ such that some boundary curve $\alpha$ in
$\partial W_n$ does not separate $M$ and $\partial W_n$ contains at
least one other component different from $\alpha$. In this case,
there exists a simple closed curve $\beta$ which intersects $\alpha$
transversally in a single point and is transverse to $\partial W_n$.
Let $W_n'$ be the union of $W_n$ and a closed regular neighborhood
of the embedded arc in $\beta \cap (M-\Int(W_n))$ whose ends points
are contained in $\alpha$ and in a second boundary component of
$\partial W_n$. The surface $W_n'$ is connected and $M-\Int(W_n')$
has no compact components because $M -\Int(W_n)$ has none. Since
$W_n'$ contains one less boundary component than $W_n$, after a
finite number of modifications of this type  to $W_n$, we obtain a
new connected surface $W_n''$ such that each boundary component of
this surface separates $M$ and $M-\Int(W_n'')$ has no compact
components. The surface $W_n''$ is a subset of some $W_{n+k}$.
Consider the new exhaustion $W_1 \subset \dots \subset
W_{n-1}\subset W_n'' \subset W_{n+k} \subset \dots$. Repeating this
argument inductively, one obtains a new  compact exhaustion  $W_1
\subset \dots \subset W_n \subset \dots$ with the desired
properties. \vskip 2mm

\begin{assertion} \label{as:four}
The exhaustion can be modified to satisfy property 3 in the
definition of simple exhaustion, and so that the exhaustion
continues to satisfy the conclusions of Assertion~\ref{as:one}
\end{assertion}
Suppose that for $k \leq n-1$, $W_{k+1}-\Int(W_k)$ satisfies
property 3 in the definition of simple exhaustion but
$W_{n+1}-\Int(W_n)$ fails to satisfy this property. One way that
$W_{n+1}-\Int(W_n)$ can  fail to satisfy this property is that
$W_{n+1}-\Int(W_n)$ consists entirely of annuli. Since $M$ has
infinite topology, there is a smallest $m >n$ such that
$W_{m}-\Int(W_n)$ has a connected component $F$ which is not an
annulus. Thus, after removing the indexed domains $W_j$, $n<j<m$,
from the exhaustion  and reindexing, we may assume that $W_{n+1}-
\Int(W_n)$ contains a compact component  $\Delta$ that is not an
annulus and which satisfies:
\begin{itemize}
\item $\Delta$ has exactly one boundary component  $\delta_1$ in $\partial W_n$;
the existence of $\delta_1$ is a consequence of
Assertion~\ref{as:one}.
\item $\Delta$ has at least one boundary component in $\partial W_{n+1}$.
\end{itemize}

After the above modification, if $W_{n+1}-\Int(W_n)$ fails to satisfy property 3, then
$| \chi(W_{n+1})|>1$, where
$\chi(\cdot)$ denotes the Euler characteristic.
Let $\{\delta_1,\delta_2, \ldots, \delta_\alpha\}$ be the components
of $\partial W_n$ and let $A_i$, $i=1,\ldots ,\alpha$, be a small
annular neighborhood of $\delta_i$ contained in $\Int(W_{n+1})$. If
the genus of $\Delta$ is positive, then there exists a compact
annulus with a handle $\Delta' \subset \Int( \Delta)$ with $\delta_1
\subset \partial \Delta'$ and $A_1 \subset \Delta'$. If the genus of
$\Delta$ is zero, there exits a pair of pants $\Delta' \subset
\Int(\Delta)$  with $\delta_1 \subset \partial \Delta'$ such that
each of the other two boundary curves of $\Delta'$ separates $M$
into two noncompact domains, and $A_1 \subset \Delta'$ . In either
case, define
$$W_{n+1}''= W_{n} \cup \Delta'\cup \left(\bigcup_{i=1}^\alpha A_i \right).$$
Observe that  $0 \leq | \chi(W_{n+1}'')| < |\chi(W_{n+1})|$. Also
note that the compact exhaustion $$W_1 \subset \cdots \subset W_n
\subset W_{n+1}'' \subset W_{n+1} \subset W_{n+2} \subset \cdots $$
satisfies Assertion~\ref{as:one} and property 3 in the definition of
simple exhaustion for levels $k \leq n$. After a smallest positive
integer $j \leq \left|\chi(W_{n+1}-\Int(W_n))\right|$ of
modifications of this sort, we arrive at the refined exhaustion:
$$W_1 \subset \cdots \subset W_n \subset W_{n+1}''  \subset W_{n+2}''
\subset \cdots \subset W_{n+j}''\subset W_{n+1}\subset \cdots,$$
such that $W_{n+1}-\Int(W_{n+j}'')$ consists of annuli. It is
straightforward to check that the new refined exhaustion
$$W_1 \subset \cdots \subset W_n \subset W_{n+1}''
\subset W_{n+2}'' \subset \cdots \subset W_{n+j-1}''\subset
W_{n+1}\subset \cdots,$$ fulfills property 3 of a simple exhaustion
through the domain $W_{n+1}$ and such that Assertion~\ref{as:one}
also holds.  Repeating these arguments inductively, we obtain an
exhaustion which satisfies property 3 in the definition of a simple
exhaustion. \vskip 4mm

An exhaustion which satisfies  Assertion~\ref{as:four}  is a simple
exhaustion and the lemma now follows.
\end{proof}

\section{Proof of the main theorems} \label{main}
In this section we prove Theorem~\ref{th:limit} in the case of open
orientable surfaces. First, we need the following definition.
\begin{definition}\label{def:limit}
Let $f\colon M \to \cd$ be a proper immersion of an open surface $M$
into a domain $\cd$ in $\rth$. We define the {\bf limit set} of an
end $e$ of $M$ as $$L(e)=\bigcap_{\a\in I}(\overline{f(E_{\a})} -
f(E_{\a})),$$ where $\{ E_{\a}\}_{\a \in I}$ is the collection of
proper subdomains of $M$ with compact boundary which represent $e$.
Notice that $L(e)$ is a compact connected set of $\partial \cd$.
\end{definition}

\begin{theorem} \label{th:first}
Let $M$ be an open orientable surface and let $\mathcal{D}$ be a
domain in $\rth$ which is either convex (possibly all $\rth$) or
bounded and smooth. Then, there exists a complete, proper minimal
immersion $f \colon M \rightarrow \mathcal{D}$. Moreover, we have:
\begin{enumerate}

\item There exists a smooth  exhaustion $\{\mathcal{D}_n \; \mid \; n \in \n \}$ of the
domain $\mathcal{D}$ such that $\{M_n=f^{-1}(\overline{\mathcal{D}_n}) \; \mid
\; n \in \n \}$ is simple exhaustion of $M$;

\item If $\mathcal{D}$ is convex, then for any simple exhaustion
$\{ M_n \; | \; n \in \n\} $ of $M$ and for any smooth exhaustion
$\{\mathcal{D}_n \; \mid \; n \in \n \}$, where $\mathcal{D}_n$, $n
\in \n$, are bounded and strictly convex\footnote{Any convex domain
admits such a exhaustion by a classical result of Minkowski.}, the
immersion $f$ can be constructed in such a way that $f(M_n)=f(M)
\cap \overline{\mathcal{D}_n}$;
\item Suppose $\mathcal{D}$ is smooth and bounded, and fix
some open subset $U \subseteq \partial \mathcal{D}$ such that $U$
has positive mean and positive Gaussian curvature, with respect to
the inward pointing normal to $\partial \cd$. Then the minimal
immersion $f:M \to \cd $ can be constructed in such a way that the
limit set of different ends of $M$ are disjoint subsets of $U$.
\end{enumerate}
\end{theorem}
\begin{proof}
In the proof of this theorem, we will distinguish three cases,
depending on the nature of the domain $\cd$.

\noindent {\bf Case 1.} $\cd $ {\em is a  general convex domain, not
necessarily bounded or smooth}.

Let $M$ be an open surface and $\mathcal{M}=\{M_1 \subset M_2
\subset  \cdots \subset M_n \subset \cdots \}$ be a simple
exhaustion of $M$. Consider $\{ \cd_n, \; n \in \n\}$ a smooth
exhaustion of $\cd$, where $\cd_n$ is bounded and strictly convex,
for all $n$. The existence of such an exhaustion is guaranteed by a
classical result of Minkowski (see \cite[\S 2.8]{mingorebulgo}).

Our purpose is to construct a sequence of minimal surfaces $\{ \Sigma_n \; | \; n \in \n\}$ with
nonempty boundary satisfying:
\begin{enumerate}[(1$_n$)]
\item $\vec 0 \in \Sigma_n$ and $\partial \Sigma_n \subset \partial \cd_n$;
\item \label{item:graph}For $i=1, \ldots,n-1$, $\Sigma_n \cap \overline{\cd}_{i}$ is a normal
graph over its projection $\Sigma_{i,n} \subset \Sigma_{i}'$, where
$\Sigma_i'$ is a larger compact minimal surface containing
$\Sigma_i$ in its interior. Furthermore,  if we write
$\Sigma_n=\{p+f_{n,i}(p) \cdot N_i(p) \; | \; p \in \Sigma_{i,n}
\}$, where $N_i$ is the Gauss map of $\Sigma_{i,n}$, then:
 \begin{enumerate}[(\ref{item:graph}$_n$- a)]
 \item $ \displaystyle | \nabla f _{n,i} | \leq \sum_{k=i+1}^n \varepsilon_k,$ and
 \item $\displaystyle \delta^H(\Sigma_n \cap \overline{\cd}_{i}, \Sigma_{i}) \leq
  \sum_{k=i+1}^n \varepsilon_k,\quad \mbox{for } i=1, \ldots,n-1.$
 \end{enumerate}
 where $\varepsilon_k>0$, for all $k$, and $\displaystyle \sum_{k=1}^\infty \varepsilon_k<1.$
\item $\dist_{\Sigma_n}(\vec 0, \partial \Sigma_n) \geq \dist_{\Sigma_1}(\vec 0, \partial \Sigma_1)+n-1;$
\end{enumerate}

The sequence $\{ \Sigma_n \; | \; n \in \n\}$ is obtained by recurrence.

In order to define the first element of the family, we consider
an analytic Jordan curve $\Gamma_1$  in $\partial \cd_1$ and we solve
the classical Plateau problem associated to this curve. The minimal disk
obtained in this way is smooth and embedded \cite{meeksyau} and it
is the first term of the sequence $\Sigma_1$.
Up to a suitable translation in $\r^3$, we can assume that $\vec 0
\in \Int(\Sigma_1) \subset \cd_1.$ It is obvious that $\Sigma_1$
satisfies Properties (1$_1)$ and (4$_1$) (notice that the other two
properties do not make sense for $n=1$.)

Assume now we have defined $\Sigma_n$, satisfying items from (1$_n$)
to (4$_n$).  We are going to construct the minimal surface
$\Sigma_{n+1}$. As the exhaustion $\mathcal{M}$ is simple, then we
know that $M_{n+1}- \Int(M_n)$ contains a unique nonannular
component $N$ which topologically is a pair of pants or an annulus
with a handle. Label $\gamma$ as the connected component of
$\partial N$ that is contained in $\partial M_n$. We label the
connected components of $\partial \Sigma_n$, $\Gamma_1, \ldots,
\Gamma_k$, in such a way that $\gamma$ maps to $\Gamma_k$ by the
homeomorphism which maps $M_n$ into $\Sigma_n$. Then, we apply
Lemma~\ref{lem:adding-ends} or \ref{lem:adding-handles} (depending
on the topology of $N$) to the data
$$\cd=\cd_n, \quad \cd'=\cd_{n+1}, \quad M=\Sigma_n. $$
Then, we obtain a family of minimal surfaces with boundary, $\Sigma_\varepsilon$, satisfying:
\begin{enumerate}[(i)]
\item $\partial \Sigma_\varepsilon \subset \partial \cd_{n+1}$ and
$\vec 0 \in \Int(\Sigma_\varepsilon);$
\item $\dist_{\Sigma_\varepsilon} (\vec 0, \partial
\Sigma_\varepsilon)>\dist_{\Sigma_n}(\vec 0, \partial \Sigma_n)+1
\geq \dist_{\Sigma_1}(\vec0, \partial \Sigma_1)+n $ (notice that
$\Sigma_n$ satisfies property (3$_n$));
\item \label{sigma-3} The surfaces $\Sigma_\varepsilon \cap \overline{\cd_n}$ are
diffeomorphic to $\Sigma_n$ and converge in the $C^\infty$ topology
to $\Sigma_n$, as $\varepsilon \to 0$. Furthermore,
$\delta^H(\Sigma_n, \Sigma_\varepsilon \cap \overline{\cd_n}) <
\varepsilon;$
\item \label{sigma-4} $\Sigma_\varepsilon - \dc_n$ consists of $k-1$ annuli whose
boundary in $\partial \cd_n$ lies in $T(\Gamma_j,\varepsilon)$,
$j=1, \ldots,k-1$, and a nonannular piece which is homeomorphic to
$N$  whose boundary in $\partial \cd_n$ is a single curve which lies
in  $T(\Gamma_k,\varepsilon)$;
\end{enumerate}
Item \eqref{sigma-3} and property (2$_n$) imply that
$\Sigma_\varepsilon \cap \overline{\cd}_i$ can be expressed as a
normal graph over its projection $\Sigma_{i,\varepsilon} \subset
\Sigma_i'$, $i=1, \ldots, n$; $\Sigma_\varepsilon \cap
\overline{\cd}_i=\{p+f_{\varepsilon,i}(p) \, N_i(p) \; | \; p \in
\Sigma_{i,\ve} \}.$ Since as $\ve\to 0$ $\Sigma_\varepsilon$
converges smoothly to $\Sigma_n$ in $\cd_n$ and $\Sigma_n$ satisfies
(2$_n$-a), then we have:
\begin{equation} \label{eq:grad}
|\nabla f_{\varepsilon,i}| < \sum_{k=i+1}^{n+1} \varepsilon_k.
\end{equation}
Moreover, if we take $\varepsilon < \varepsilon_{n+1}$, then item
\eqref{sigma-3} and property (2$_n$-2) implies that
\begin{equation} \label{eq:hauss}
\delta^H \left(\Sigma_\varepsilon \cap \overline{\cd}_i, \Sigma_i
\right)< \sum_{k=i+1}^{n+1} \varepsilon_k;
\end{equation}
here we have also used the triangle inequality for $\delta^H$.

Then, we define $\Sigma_{n+1}\df\Sigma_\varepsilon$, where
$\varepsilon$ is chosen small enough in order to satisfy
\eqref{eq:grad} and \eqref{eq:hauss}. It is
clear that $\Sigma_{n+1}$ so defined fulfills (1$_{n+1}$), (2$_{n+1}$) and (3$_{n+1}$).

Now, we have constructed our sequence of minimal surfaces $\{
\Sigma_n \}_{n \in \n}$. Taking into account properties (2$_n$), for
$n \in \n$, and using Ascoli-Arzela's theorem, we deduce that the
sequence of surfaces $\{ \Sigma_n \}_{n \in \n}$ converges to an
open immersed minimal surface $\Sigma$ in the $C^m$ topology, for
all $m \in \n$. Moreover, $\Sigma \cap \overline{\cd}_i$ is a normal
graph over its projection $\Sigma_{i,\infty} \subset \Sigma_i'$, for all $i$, and the
norm of the gradient of the graphing functions its at most $1$ (see
properties (2$_n$-a)).

Finally, we check that $\Sigma$ satisfies all the statements in the
theorem.

\noindent $\bullet$ {\em $\Sigma$ is properly immersed in $\cd$.} To
see this, we consider $K \subset \cd$ a compact subset. We have to
prove that $\Sigma \cap K$ is compact. As $\{ \cd_n \; : \; n \in
\n\}$ is an exhaustion of $\cd$, then we know that there exists $n_0
\in \n$ such that $K \subset \cd_{n_0}$. We also know that $\Sigma
\cap \overline{\cd}_{n_0}$ is a graph over $\Sigma_{n_0}$ which is
compact. Therefore $\Sigma \cap \overline{\cd}_{n_0}$ is compact and
$\Sigma \cap K$ is a closed subset compact set, consequently $\Sigma
\cap K$ is compact.

\noindent $\bullet$ {\em $\Sigma$ is complete.} Consider the compact
exhaustion $\Sigma \cap \cd_n$ of $\Sigma$ and note that $\Sigma
\cap \cd_n$ is quasi isometric to $\Sigma_{n, \infty}$ with respect
to constants that are independent of $n$. Then properties (3$_n$),
$n \in \n$,  trivially imply that $\Sigma$ is complete.

\noindent $\bullet$ {\em $\Sigma$ is homeomorphic to $M$.} If we
consider the exhaustions $\{\Sigma \cap \overline{\cd}_n \; | \; n
\in \n\}$ of $\Sigma$ and $\{ M_n \; | \; n \in \n \}$ of $M$, then
we know (from the way in which we have constructed $\Sigma$) that
$\Sigma \cap \overline{\cd}_n$ is homeomorphic to $M_n$. Label this
homeomorphism as $f_n \colon \Sigma \cap \overline{\cd}_n
\rightarrow M_n$.
\begin{center}
\begin{picture}(80,80)(0,0)
\put(8,2){\mbox{$M_i$}}
\put(75,2){\mbox{$M_n$}}
\put(77,50){\mbox{$\Sigma\cap \overline{\cd}_n$}}
\put(0,50){\mbox{$\Sigma\cap \overline{\cd}_i$}}
\put(25,6){\vector(1,0){45}}
\put(35,54){\vector(1,0){35}}
\put(14,45){\vector(0,-1){30}}
\put(81,45){\vector(0,-1){30}}
\put(18,28){\mbox{${f_n|}_{\Sigma\cap \overline{\cd}_i}$}}
\put(84,28){\mbox{$f_n$}}
\put(45,10){i}
\put(48,58){i}
\end{picture}
\end{center}
Moreover, we have that ${f_n|}_{\Sigma\cap \overline{\cd}_i}$ is
also a homeomorphism between $\Sigma\cap \overline{\cd}_i$ and $M_i$
which coincides with the corresponding homeomorphism $f_i$.  Then,
after taking the limit as $n\to \infty$, we conclude that $\Sigma$
and $M$ are homeomorphic.

\noindent {\bf Case 2.} $\cd $ {\em is a smooth strictly convex domain}.

First of all, we can assume, up to a suitable shrinking of $\cd$, that
$\kappa_1(\partial \cd)=1$. This time the proof is slightly
different from the previous case. Our aim is to create a sequence:
$$\Theta_n=\{t_n, \varepsilon_n, \delta_n, \cd_n, \Sigma_n \}_{n \in \n}, $$
where:
\begin{itemize}
\item $\{t_n\}_{n \in \n}$, $\{\varepsilon_n\}_{n \in \n}$, $\{\delta_n\}_{n \in \n}$,
are sequences of real numbers decreasing to $0$. Moreover,
$\displaystyle \sum_{n=i+1}^\infty \varepsilon_n < \delta_i$ for any
$i \in \n.$
\item $\cd_n\df \cd_{-t_n}$  is the convex domain parallel to $\cd$ at distance $t_n$.
\item $\Sigma_n$ is a compact, connected, minimal surface with nonempty boundary.
\end{itemize}

This sequence can be constructed in such a way so that it satisfies:
\begin{enumerate}[(1$_n$)]
\item $\vec 0 \in \Sigma_n$ and $\partial \Sigma_n \subset \partial \cd$;
\item \label{item:graph1}For $i=1, \ldots,n-1$, $\Sigma_n \cap \overline{\cd}_{i}$ is a normal
graph over its projection $\Sigma_{i,n} \subset \Sigma_{i}'$, where
$\Sigma_i'$ is a larger compact minimal surface containing
$\Sigma_i$ in its interior. Furthermore,  if we write
$\Sigma_n=\{p+f_{n,i}(p) \cdot N_i(p) \; | \; p \in \Sigma_{i,n}
\}$, where $N_i$ is the Gauss map of $\Sigma_{i,n}$, then:
\begin{enumerate}[(\ref{item:graph}$_n$- a)]
\item $ \displaystyle | \nabla f _{n,i} | \leq \sum_{k=i+1}^n \varepsilon_k,$ and
\item $\displaystyle \delta^H(\Sigma_n , \Sigma_{i}) \leq
\sum_{k=i+1}^n \varepsilon_k,\quad \mbox{for } i=1, \ldots,n-1.$
\end{enumerate}
where $\varepsilon_k>0$, for all $k$, and $\displaystyle
\sum_{k=1}^\infty \varepsilon_k<1.$
\item $\dist_{\Sigma_n}(\vec 0, \partial \Sigma_n) \geq \dist_{\Sigma_1}(\vec 0, \partial \Sigma_1)+n-1;$
\item Let $2 \, \delta_i \df \min_{j \neq k} \dist_{\rth}(C_j, C_k)$, where $C_j$ are
the connected components of $\Sigma_i\cap(\overline{\cd}-\cd_i)$. If
there is only one component in $\Sigma_i\cap(\overline{\cd}-\cd_i)$,
then we define $\delta_i \df 1/2$. If $C$ and $C'$ are two
different connected components of $\Sigma_n \cap
\left(\overline{\cd}-\cd_i \right)$, then the distance
$\dist_{\rth}(C,C')>\delta_i.$
\end{enumerate}

The sequence $\{\Theta_n\}_{n \in \n}$ is obtained in a recurrent
way. In order to define $\Sigma_1$, we consider  an analytic Jordan
curve $\Gamma_1$ in $\partial \cd$. We solve the Plateau problem for
this curve and let $\Sigma_1$ be the solution minimal disk. Up to a
translation in $\rth$, we can assume that $\vec 0 \in \Int(\Sigma_1)
\subset \cd.$

Suppose that we have constructed the term $\Theta_n$ in the
sequence. The idea is to apply Lemma~\ref{lem:adding-ends} or
 Lemma~\ref{lem:adding-handles} (depending on the topology of
$M_{n+1}-\Int(M_n)$) to produce the next minimal surface
$\Sigma_{n+1}$, like in the proof of Case 1. However, this time we
have to be more careful. First, we take $t_{n+1} \in (0,t_n)$ small
enough so that:
\begin{itemize}
\item $\Sigma_n$ intersects $\partial \cd_{-t_{n+1}}$ transversally and
$\Sigma_n \cap \overline{\cd}_{-t_{n+1}}$ contains a connected
component $\widehat \Sigma_n$ with the same topological type than
$\Sigma_n$ and satisfies  $$\dist_{\widehat \Sigma_n}(\vec 0,\partial
\widehat \Sigma_n)\geq \dist_{\Sigma_1}(\vec 0, \partial
\Sigma_1)+n-1.$$

\item The constant $C(\varepsilon',\cd_{-t_{n+1}},\cd)=\varepsilon'+
\sqrt{(t_{n+1}+2 \varepsilon'+1)^2-1}<\varepsilon_{n+1}$ for
$\varepsilon'$ sufficiently small.
\end{itemize}
Then apply one of the lemmas to the data $ \Sigma_n$, $\cd_{n+1}$
and $\cd$. In this way, we obtain the new immersion $\Sigma_{n+1}$
satisfying properties (1$_{n+1}$) to (4$_{n+1}$). Let us check
(4$_{n+1}$). Take $ C$ and $C'$ two components of $\Sigma_{n+1}\cap
(\overline{\cd}-\cd_i)$. Then $C$ and $C'$ lie in tubular
neighborhoods of radius $\sum_{k=i+1}^{n+1}\varepsilon_k$ of some
components of $\Sigma_{i}\cap (\overline{\cd}-\cd_i)$, that we label
$\widetilde C$ and $\widetilde C',$ respectively. Then one has
$$ \dist_{\rth}(C,C') \geq \dist_{\rth}(\widetilde C, \widetilde C') -
\sum_{k=i+1}^{n+1}\varepsilon_k > 2 \delta_i-
\sum_{k=i+1}^{n+1}\varepsilon_k >\delta_i.$$

If $\sum_{k={n+2}}^\infty \varepsilon_k \geq \delta_{n+1}$, then we
modify the sequence $\{\varepsilon_n\}_{n\in \n}$ as follows:
\begin{itemize}
\item $\varepsilon_k'=\varepsilon_k$, for $k=1, \ldots, n+1$;
\item $\varepsilon_k'=\delta_{n+1} \, \varepsilon_k$, for $k>n+1$.
\end{itemize}
At this point in the proof, we have obtained a sequence of compact
minimal surfaces $\{ \Sigma_n\}_{n \in \n}$ with regular boundary in
$\partial \cd$, whose interiors converge smoothly on compact sets of
$\cd$ to a complete minimal surface $\Sigma$, properly immersed in
$\cd$. As in the previous step, we have that $\Sigma \cap
\overline{\cd_i}$ is homeomorphic to $M_i$, for all $i \in \n$, and
for each $i \in \n$, $\Sigma \cap \overline{\cd_i}$ is a small graph
over $\Sigma_i$. Furthermore, properties (4$_n$), $n \in \n$, imply
that  the distances between any two components of $\Sigma \cap (\cd
-\cd_i)$ are larger than $\delta_i$. Note that two different ends
$e_1$, $e_2$ of $\Sigma$ can be represented by distinct components
$C_1$, $C_2$ of $\Sigma-\cd_j$, for some $j$ sufficiently large. By
Definition~\ref{def:limit},  the distance between $L(e_1)$ and
$L(e_2)$ is at least equal to the distance between $C_1$ and $C_2$
which is greater than $\delta_j$. This completes the proof of Case
2.

\noindent {\bf Case 3}. {\em $\cd$ is a smooth bounded domain.}

 In this
case we take $U$ to be an open disk in $\partial \cd$ so that the
principal curvatures with  respect to the inner pointing normals are
positive and bounded away from zero. Then, it is possible to find a
smooth convex domain $\cd_U \subset \cd$ with $U \subset \partial
\cd_U$. Then we consider the curve $\Gamma_1 \subset U$ as in the
previous case, and we solve the classical Plateau problem to obtain
a compact minimal disk $\Sigma_1$. We take the series
$\sum_{k=1}^\infty \ve_k$ to satisfy:
$$\sum_{k=1}^\infty \ve_k<\frac 12 \dist_{\rth}(\Sigma_1, \partial \cd-U).$$
Thus, we apply Case 2 to obtain a complete minimal surface $\Sigma$
satisfying the conclusions of the theorem for the domain $\cd'$ and
the limit set of $\Sigma$ is contained in $U$. Then the surface
$\Sigma$ is also properly immersed in $\cd.$ This concludes the
proof of the theorem.
\end{proof}

Suppose $M$ is a proper minimally  immersed open surface in $\rth$
and passes through the origin. After a small translation of $M$
assume that $M$ is transverse to the boundary sphere of the balls
$\b (n)$ of radius $n$, $n\in \n$.  Then the maximum principle
implies that the exhaustion $\{M_n=M\cap \b (n)\}$ of $M$ is a
smooth compact exhaustion where for all $$n\in \n, \quad M-
\Int(M_n) \mbox{ \; has no compact components.}$$ We will call a
smooth compact exhaustion $M_1 \subset  M_2 \subset \cdots \subset
M_n\subset \cdots$  {\it admissible} if it satisfies the above
property.  The next result is an immediate corollary of
Theorem~\ref{th:first}.
\begin{theorem} \label{th:rth}
Let $M$ be an open orientable surface with an admissible exhaustion
$M_1 \subset M_2 \subset \cdots M_n \subset \cdots$. There exists a
proper minimal immersion $f\colon M \longrightarrow \rth$ satisfying
 $f(M_n) = f(M) \cap \b(n)$.
 \end{theorem}

The question concerning the existence of complete proper minimal
surfaces
 in the unit ball $\b(1)$ such that the limit sets are the entire unit sphere $\s^2(1)$
 was proposed to the second author by Nadirashvili in 2004. The
 techniques used to prove Theorem~\ref{th:first} allow us to give a positive
 answer to this former question.

 \begin{proposition} \label{prop:nadiq} Let $M$ be an open orientable surface and $\cd$ a
 convex open domain. Then there exists a complete proper minimal
 immersion $f:M \to \cd$ such that the limit set of $f(M)$ is $\partial \cd$.
 \end{proposition}
 The proof of the above proposition consists of a suitable use of the
 bridge principle in the proof of Lemmas~\ref{lem:adding-ends} and
 \ref{lem:adding-handles}. In this case the curve $\Gamma$ used in
 Step 3 in both lemmas is substituted by a smooth arc in
 $\partial \cd'$ which is  $\varepsilon$ close to every point of $\partial \cd'$.
 With these new versions of the lemmas we can modify the proof of Case 1
 (when $\partial \cd$ is convex)  as follows:
  we construct the sequence
 $\{ \Sigma_n\}_{n \in \n}$  in such a way that $\partial \Sigma_n$ is
 $\frac 1 n$ close to every point in $\partial
 \cd_n$. So, the limit immersion $\Sigma$ would satisfy that its limit set
 $L(\Sigma)$ is $\partial \cd.$

 As a consequence of Proposition~\ref{prop:nadiq}, we obtain the following
 result.
 \begin{corollary} Any convex domain of $\r^3$ is the convex hull of
 some complete minimal surface.
 \end{corollary}

\section{Nonorientable minimal surfaces} \label{sec:no} The main goal of
this section is to develop the necessary theory for dealing with
complete, properly immersed or embedded, nonorientable minimal
surfaces in domains in $\rth$. First we explain how to modify
arguments in the proof of the Density Theorem in \cite{density} to
the case of nonorientable surfaces, i.e., given a compact
nonorientable surface $M$, we describe how to approximate it by a
complete, nonorientable hyperbolic surface $\wt{M}$ which is
homeomorphic to the interior of $M$. Once this generalization of the
Density Theorem is seen to hold, we apply it to prove that Theorem
\ref{th:first} holds for nonorientable surfaces, which then
completes the proof of Theorem~\ref{th:limit} stated in the
Introduction.

Since one of the goals in the Embedded Calabi-Yau Conjecture is to
construct nonorientable, properly embedded minimal surfaces in
bounded domains of $\rth$, we construct in Section~\ref{sec:6.3}
complete, proper minimal immersions of any open surface $M$ with a
finite number of nonorientable ends into some certain smooth
nonsimply connected domain such that distinct ends of $M$ have
disjoint limit sets and such that the immersed surface is properly
isotopic to a proper (incomplete) minimal embedding of $M$ in the
domain. In Example~\ref{ex:universal}, we construct a bounded domain
$\cd_\infty$ in $\rth$ which is smooth except at one point
$p_\infty$  and has the property that every open surface $M$ admits
a complete, proper minimal immersion  $f:M\to\cd_\infty$ which can
be closely approximated in the Hausdorff distance by a proper,
noncomplete, minimal embedding of $M$ in  $\cd_\infty$.

\subsection{Density theorems for nonorientable minimal surfaces} \label{subsec:non}
The results contained in \cite{density} remain true when the minimal
surfaces involved in the construction are nonorientable. In order to
obtain a result similar to Lemma~\ref{lem:afm} in the nonorientable
setting, we work with the orientable double covering. But then all
the machinery must be adapted in order to be compatible with the
antiholomorphic involution of the change of sheet in the orientable
covering. In verifying this construction, there are three points
that are nontrivial and they are explained in paragraphs
\ref{subsubsec:Runge}, \ref{subsubsec:I}, and \ref{subsubsec:h}
below.

First, we need some notation. Let  $M'$  denote a connected
compact Riemann surface of genus $\sigma \in\n \cup \{0\}.$ Let $ I:M'
\rightarrow M'$ be an antiholomorphic involution without fixed
points. Then, the surface $\widetilde M' \df M'/\langle I \rangle$ is a
compact connected nonorientable surface.

For $\textsc{e}\in\n$, consider $\d_1,\ldots,\d_\textsc{e}\subset M'$
open disks so that $\{\g_i\df\partial \d_i, \; \; i=1, \ldots, \textsc{e}\}$ are piecewise
smooth Jordan curves and $\overline{\d}_i\cap \overline{\d}_j=\emptyset$
for all $i\neq j$.

\begin{definition}\label{multicycle}
Each curve $\g_i$ will be called a cycle on $M'$ and the family
$\mathcal{ J}=\{\g_1,\ldots,\g_\textsc{e}\}$ will be called a {\bf
multicycle} on $M'$. We denote by $\intc(\g_i)$ the disk $\d_i$, for
$i=1,\ldots, \textsc{e}.$ We also define $M(\mathcal{
J})=M'-\cup_{i=1}^\textsc{e} \overline{\intc(\g_i)}$. Notice that
$M(\mathcal{ J})$ is always connected.

We will say that $\cj$ is {\bf invariant under $I$} iff for any disk
$\d_i$ there exist another disk in the family $\d_j$ such that
$I(\d_i)=\d_j.$ Observe that $i \neq j$ and so the number of cycles
in $\cj$ is even in this case.
\end{definition}

Given $\mathcal{ J}=\{\g_1,\ldots,\g_\textsc{e}\}$ and $\mathcal{
J}'=\{\g_1',\ldots,\g_\textsc{e}'\}$ two multicycles in $M'$, we
write $\mathcal{ J}'< \mathcal{ J}$ if $\overline{\intc(\g_i)}
\subset \intc (\g_i')$ for $i=1,\ldots, \textsc{e}.$ Observe that
$\mathcal{J}'<\mathcal{J}$ implies $\overline{M(\mathcal{ J}')}
\subset M(\mathcal{ J})$.

\subsubsection{Runge functions on nonorientable minimal surfaces} \label{subsubsec:Runge}
Runge-type theorems are crucial in obtaining the theorems for
orientable surfaces obtained previously in \cite{density}. So, the
first step in the proof of Lemma~\ref{lem:afm} in the nonorientable
case consists of proving a suitable Runge theorem for nonorientable
minimal surfaces. To be more precise, we need the following.

\begin{lemma}\label{lem:rungeno}
Let $\cj$ be a multicycle in $M'$which is invariant under $I$ and
let $F: \overline{M(\cj)} \rightarrow \r^3$ be a nonorientable
minimal immersion with Weierstrass data $(g,\Phi_3)$\footnote{Recall
that $g \circ I=-1/\overline{g}$, $I^*\Phi_3=\overline{\Phi_3}$.}.
Consider $K_1$ and $K_2$ two disjoint compact sets in $M(\cj)$ and
$\Delta \subset M'$ satisfying:
\begin{enumerate}[(a)]
\item There exists a basis of the homology of $M(J)$ contained in $K_2$ and $I(K_2)=K_2$;
\item $\overline{\Delta} \subset M'-(K_1 \cup I(K_1) \cup K_2)$ and $I(\Delta)=\Delta$;
\item $\Delta$ has a point in each connected component of $M'-(K_1 \cup I(K_1) \cup K_2 ).$
\end{enumerate}
Then, for any $m \in \n$ and any $t>0$, there exists a holomorphic
function without zeros $H: M(\cj)-\Delta \rightarrow \c$ such that:
\begin{enumerate}[(1)]
\item $H\circ I=1/\overline{H};$
\item $|H-t|<1/m$ in $K_1;$
\item $|H-1|<1/m$ in $K_2;$
\item The nonorientable minimal immersion given by the Weierstrass
data $\widetilde g\df g/H$ and $\widetilde \Phi_3:=\Phi_3$ is
well-defined (has no real periods.)
\end{enumerate}
\end{lemma}
\begin{proof}
If $\sigma$ represents the genus of $M'$ and $2 \E$ is the number of
cycles in $\cj$, notice that the dimension of $H_1(M(\cj),\r)$ is $2
\sigma+2 \E-1.$

\begin{assertion}
There exists a basis for the first real homology group of $M(\cj)$ $$B=
\{\gamma_1, \ldots,\gamma_{\sigma+\E},\Gamma_1, \ldots,
\Gamma_{\sigma+\E-1}\},$$ which is contained in $K_2$ and satisfies:
\begin{itemize}
\item $I_*(\gamma_j)=\gamma_j$, for $j=1, \ldots , \sigma+\E$,
\item $I_*(\Gamma_j)=-\Gamma_j$, for $j=1, \ldots , \sigma+\E-1$.
\end{itemize}
\end{assertion}
The proof of this assertion is a standard  topological argument that
can be found in \cite{LMM-Nonorientable}, for instance.
\begin{assertion}
If $\tau$ is a holomorphic differential in $M(\cj)$ satisfying
$I^*(\tau)= \overline{\tau}$, then $\re \left(\int_\gamma
\tau\right)=0,$ for all $\gamma$ in $H_1(M(\cj),\r)$ if and only if
$\int_{\gamma_j} \tau =0,$ for all $j=1, \ldots , \sigma+\E$.

In addition, if $\tau$ is holomorphic on $M'$, then $\tau=0$ if and
only if $\int_{\gamma_j} \tau =0,$ for all $j=1, \ldots ,
\sigma+\E$.
\end{assertion}
\begin{proof}
The proof of the first part of this claim is straightforward. For
the second part, take into account that a holomorphic differential
on a compact Riemann surface is zero if and only if it has imaginary
periods.
\end{proof}
\begin{assertion} \label{homologia}
Consider $(b_1, \ldots,b_{\sigma+\E} )\in \r^{\sigma+\E}-\{\vec 0\}$
and $c=\displaystyle \sum_{j=1}^{\sigma+\E} b_j \cdot \gamma_j$,
then there exists a holomorphic differential on $\overline{M(\cj)}$
satisfying $I^*\tau=-\overline{\tau}$ and $\int_c
\tau\not=0.$

Furthermore, given $L$ an integral divisor in $M'$, invariant under
$I$ and with $\mbox{\rm supp}(L) \subset \overline{M(\cj)}$, then
$\tau$ can be chosen in such a way that $(\tau)_0 \geq L$, where $(
\cdot)_0$ means the divisor of zeros.
\end{assertion}
\begin{proof}
The first holomorphic De Rham cohomology group, $H_{\rm
hol}^1(\overline{M(\cj)})$ is a complex vector space of dimension
$\varrho.$ If we define $F \colon H_{\rm hol}^1(\overline{M(\cj)})
\longrightarrow H_{\rm hol}^1(\overline{M(\cj)})$
$$F([\omega] ) \df \left[\overline{I^* \, (\omega)} \right],$$
then $F$ is a (real) linear involution of $H_{\rm
hol}^1(\overline{M(\cj)})$. Hence, $H_{\rm
hol}^1(\overline{M(\cj)})=V^+ \oplus V^-$, where $V^+=\{[\omega] \;
| \; F([\omega])=[\omega] \}$ and $ V^-=\{[\omega] \; | \;
F([\omega])=-[\omega] \}.$ Moreover, the linear map $[\omega]
\mapsto [ \ri \, \omega]$ establishes an isomorphism between $V^+$
and $V^-$. Then, we have that the real dimension $\dim_\r
V^+=\varrho.$ So, the linear map:
$$T: V^- \longrightarrow \left(\ri \cdot \r^{\sigma+\E} \right)
\times \r^{\sigma+\E-1}$$
$$T([\psi])= \left( \int_{\gamma_1} \psi, \ldots,
\int_{\gamma_{\sigma+\E}}\psi, \int_{\Gamma_1} \psi, \ldots,
\int_{\Gamma_{\sigma+\E-1}}\psi\right),$$ is an isomorphism where
$\ri=\sqrt{-1}$. In particular, there exists $[\psi]$ in $V^-$ such
that
$$T([\psi]) \notin \left\{(z_1, \ldots,z_{\sigma+\E}, w_1, \ldots ,
w_{\sigma+\E-1}) \in \left(\ri \cdot \r^{\sigma+\E} \right) \times
\r^{\sigma+\E-1} \; | \;  \sum_{j=1}^{\sigma+\E} b_j z_j=0
\right\}.$$ Hence $\im \left(\int_c \psi \right) \neq 0$. Now, using
Claim 3.2 in \cite{density}, we can prove the existence of a
holomorphic differential on $\overline{M(\cj)}$, $\widetilde \psi$,
with the same periods as $\psi$ and such that $(\widetilde \psi)_0
\geq L.$ Then, we define the $1$-form
  $\tau \df
\frac 12 \left(\widetilde \psi - \overline{I^*(\widetilde
\psi)}\right)$. From the definition, it is clear that
$I^*(\tau)=-\overline{\tau}$ and $(\tau)_0 \geq L$. Moreover, as
$\psi$ and $\widetilde \psi$ have the same periods, one has:
$$
\int_c \tau= \frac 12 \left( \int_c \widetilde \psi -
\overline{\int_c \widetilde \psi} \right)= \ri \im \left( \int_c
\widetilde \psi \right)= \ri \im \left( \int_c \psi \right) \neq 0.
$$
\end{proof}

From this point on in the proof, we can follow the proof of Lemma 1
in \cite{LMM-Nonorientable} to obtain the existence of the function
$H$ satisfying all the assertions in the lemma. For completeness, we
include a sketch of this proof.


\begin{assertion} \label{sobre}
Let $\mathcal{H}^-\left( \overline{M(\cj)} \right)$ be the {\em
real} vector space of the holomorphic functions $t:\overline{M(\cj)}
\rightarrow \c$, satisfying $t \circ I=-\overline{t}.$ Then the
linear map $F: \mathcal{H}^- \left( \overline{M(\cj)} \right)
\rightarrow \r^{2(\sigma+\E)}$, given by:
$$F(t)= \left( \int_{\gamma_j} t \; \Phi_3 \left(\frac1g+g \right) ,  -i \, \int_{\gamma_j} t \;
\Phi_3 \left(\frac1g-g \right) \right)_{j=1, \ldots,\sigma+\E}$$
is surjective.
\end{assertion}
\begin{proof}
We proceed by contradiction. Assume $F$ is not onto. Then, there
exist
$$(\vartheta_1, \ldots, \vartheta_{\sigma+\E},\mu_1,
\ldots,\mu_{\sigma+\E}) \in \r^{2(\sigma+\E)}- \{ (0, \ldots,0 )
\},$$ such that:
\begin{equation}\label{biblioteca}
\sum_{j=1}^{\sigma+\E} \left[ \vartheta_j \int_{\gamma_j} t \; \Phi_3
\left(\frac1g+g \right) - i \, \mu_j \int_{\gamma_j} t \; \Phi_3 \left(\frac1g-g \right) \right]
=0\quad \forall t\in \mathcal{ H}^- \left( \overline{M(\cj)} \right).
\end{equation}

Assertion~\ref{homologia} guarantees the existence of a differential
$\tau$ satisfying
\begin{enumerate}[(i)]
\item $(\tau)_0\geq {\left(\left(\left(\frac1g+g\right)\Phi_3\right)_{|_{\overline{M(\cj)}}}
\right)_0}^2 \, \left( \left(d \, \left(\frac{1-g^2}{1+g^2}\right)\right)_{|_{\overline {M(\cj)}}} \right)_0$,
\item $\displaystyle - i \,\sum_{j=1}^{\sigma+\E} \mu_j \int_{\gamma_j}\tau \not=0$,
\item $I^\ast \tau=-\overline\tau$.
\end{enumerate}
Let us define $\displaystyle y\df \frac \tau{d \,
\left(\frac{1-g^2}{1+g^2}\right)}, \mbox{ and } t \df \frac{d(y)
}{\left(\frac 1g+g \right)\Phi_3}.$ Taking the choice of $\tau$ into
account, the function $t$ belongs to $\mathcal{H}^-\left(
\overline{M(\cj)} \right)$. In this case and after integrating by
parts, (\ref{biblioteca}) becomes $-i \,\sum_{j=1}^{\sigma+\E} \mu_j
\int_{\gamma_j}\tau =0$, which is absurd. This contradiction proves
the claim.
\end{proof}

Using the previous claim we infer the existence of
$\{t_1,\ldots,t_{2 (\sigma+\E)}\}\subset \mathcal{H}^- \left(
\overline{M(\cj)} \right)$ such that $\det(F(t_1),\ldots,
F(t_{2(\sigma+\E)}))\not=0$. Up to changing $t_i \leftrightarrow
t_i/x$, $x>0$ large enough, we can assume that
\begin{equation} \label{por}
\left|\exp \left(\sum_{i=1}^{2(\sigma+\E)} x_it_i(p)\right)-1 \right|<1/(2 m),
\end{equation}
$  \forall (x_1,\ldots,x_{2(\sigma+\E)})\in\r^{2(\sigma+\E)}, \; |x_i|<1, \;
i=1,\ldots,2(\sigma+\E), \quad \forall p\in\overline{M(\cj)}.$
\begin{assertion}
For each $n \in \n$, there is $t_0^n \in H^-(\overline{M(\cj)})$ such that:
\begin{enumerate}[(i)]
\item $|t_0^n-n|<1/n$ in $K_1$ (and so $|t_0^n+n|<1/n$ in $I(K_1)$),
\item $|t_0^n|<1/n $ in $K_2$.
\end{enumerate}
\end{assertion}
\begin{proof}
Given $n\in\n$, we apply a Runge-type theorem on $\overline{M}$, see
\cite[Theorem 10]{royden}, and obtain a holomorphic function
$T_0^n:\overline{M(\cj)}\rightarrow\c$ satisfying
\begin{itemize}
\item $|T_0^n-n|<1/n$ in $K_1$,
\item $|T_0^n+n|<1/n$ in $I(K_1)$,
\item $|T_0^n|<1/n$ in $K_2$.
\end{itemize}
We take $t_0^n = \frac12 (T_0^n-\overline{T_0^n \circ I}).$ From
this, it is trivial to check Properties {\em (i)} and {\em (ii)}.
\end{proof}

For $\Theta=(\lambda_0, \ldots ,\lambda_{2(\sigma+\E)}) \in
\r^{2(\sigma+\E)+1}$, we define
$$h^{\Theta , n}(p) \df \exp \left[ \lambda_0 \, t_0^n(p)+
\sum_{j=1}^{2 (\sigma+\E)} \lambda_j \; t_j(p) \right],\quad \forall
p \in \overline{M(\cj)}.$$ Label $g^{\Theta , n}=g/h^{\Theta , n}$
and $\Phi_3^{\Theta,n}=\Phi_3$. As $\left\{{t^n_0}_{|_{K_2}}
\right\}_{n \in \n}$ is uniformly bounded, then, up to a
subsequence, we have $\left\{{t^n_0}_{|_{K_2}} \right\} \to
t_0^\infty \equiv 0$, uniformly on $K_2$. We also define on $K_2$
the Weierstrass data $g^{\Theta , \infty}=g/h^{\Theta , \infty}$,
$\Phi_3^{\Theta,\infty}=\Phi_3$, where
$$h^{\Theta , \infty}(p) \df \exp \left[\sum_{j=1}^{2(\sigma+\E)}
\lambda_j \; t_j(p) \right],\quad \forall p \in  K_2.$$ Observe that
third Weierstrass differential of the aforementioned holomorphic
data has no real periods. Therefore, we must only consider the
period problem associated to $\Phi_j^{\Theta,n}$, $j=1,2$. To do
this, we define the period map $\mathcal{P}_n: \r^{2(\sigma+\E)+1}
\rightarrow \r^{2(\sigma+\E)}$, $n \in \n \cup \{\infty \}$;
$$\mathcal{P}_n(\Theta)= \left(  \int_{\gamma_j} \Phi_1^{\Theta,n}  , \int_{\gamma_j}
\Phi_2^{\Theta,n} \right)_{j=1, \ldots,\sigma+\E}. $$ Since the
initial immersion $X$ is well-defined, then one has $\mathcal{
P}_n(0)=0,$ $\forall n \in \n \cup \{ \infty \}.$ Moreover, it is
not hard to check that
$$\jj(\mathcal{P}_n)(0)=\det(F(t_1), \ldots,F(t_{2(\sigma+\E)})) \neq 0,
\quad \forall n  \in \n \cup \{ \infty \}.$$

Applying the Implicit Function Theorem to the map $\mathcal{ P}_n$
at $0 \in [-\epsilon,\epsilon \,] \times \overline{B}(0,r)$, we get
an smooth function $L_n: I_n \rightarrow \r^{2(\sigma+\E)}$
satisfying $\mathcal{ P}_n(\lambda_0,L_n(\lambda_0))=0$, $\forall
\lambda_0 \in I_n,$ where $I_n$ is a maximal open interval
containing $0$ (here, maximal means that $L_n$ can not be regularly
extended beyond $I_n$).

We next check that the supremum $\epsilon_n$ of the connected
component of $L_n^{-1}(\overline B(0,r)) \cap[0,\epsilon]$
containing $\lambda_0=0$  belongs to $I_n$. Indeed, take a sequence
$\{\lambda_0^k \}_{k \in \n}\nearrow \epsilon_n$. As
$\{L_n(\lambda_0^k) \} \subset \overline{B}(0,r)$, then, up to a
subsequence, $\{L_n(\lambda_0^k) \}_{k \in \n} \to \Lambda_n \in
\overline{B}(0,r)$. Taking into account that $\jj
(\mathcal{P}_n)(\epsilon_n, \Lambda_n) \neq 0$, the local unicity of
the curve $(\lambda_0,L_n(\lambda_0))$ around the point
$(\epsilon_n, \Lambda_n)$, and the maximality of $I_n$, we infer
that $\epsilon_n \in I_n$. Therefore, either $\epsilon_n=\epsilon$,
or $L_n(\epsilon_n)=\Lambda_n \in \partial (B(0,r))$.

We will now see that $\epsilon_0 \df \liminf \{ \epsilon_n \}>0$.
Otherwise, there would be a subsequence $\{ \epsilon_n \} \to 0$.
Without loss of generality,  $\epsilon_n<\epsilon$, $\forall n \in
\n$, and so $\Lambda_n \in \partial(B(0,r))$, $\forall n \in \n$. Up
to a subsequence, $\{ \Lambda_n \} \to \Lambda_\infty \in
\partial(B(0,r))$. The fact $\mathcal{ P}_\infty(0,0)= \mathcal{
P}_\infty(0,\Lambda_\infty)=0$ would contradict the injectivity of
$\mathcal{ P}_\infty(0, \cdot)$ in $\overline{B}(0,r)$. Hence the
function $L_n :[0, \epsilon_0] \rightarrow \overline{B}(0,r)$ is
well-defined, $\forall n \geq n_0$, $n_0$ large enough.

Label $(\lambda_1^n, \ldots, \lambda_{2( \sigma+\E)}^n)=
L_n(\epsilon_0).$ From (\ref{por}) we have $|\exp[ \sum_{j=1}^{2(
\sigma+\E)} \lambda_j^n t_j]-1 | <1/(2 m)$ on $D\left(
\overline{\pg} \right)$. Hence, if $n$ ($\geq n_0$) is large enough,
the function:
$$H(z) \df \exp \left[ \epsilon_0 \, t_0^n(z)+\sum_{j=1}^{2( \sigma+\E)}
\lambda^n_j \; t_j(z) \right]$$ satisfies items {\em 1}, {\em 2} and
{\em 3} in Lemma~\ref{lem:rungeno}. Since the period function
$\mathcal{ P}_n$ vanishes at $\Theta_n=(\epsilon_0,\lambda_1^n,
\ldots, \lambda_{2(\sigma+\E)}^n)$, then the minimal immersion
$\widetilde F$ associated to the Weierstrass data $g^{\Theta_n,n}$,
$\Phi_3^{\Theta_n,n}=\Phi_3$ is well-defined. This proves item {\em
4} in the lemma.


\end{proof}

\subsubsection{The existence of a holomorphic differential without zeros} \label{subsubsec:I}

In the paper \cite{density}, the existence of a holomorphic $1$-form
without zeros $\omega$ on $M(\cj_0)$ is used over and over again,
for a given multicycle $\cj_0$. In our new setting, we need the
following related result:
\begin{lemma} \label{th:I}
Given $ \cj_0$ a multicycle in $M'$, which is invariant under $I$,
there exists a holomorphic $1$-form $\omega'$ in $M(\cj_0)$, without
zeros, and satisfying $I^*(\omega')=\overline{ \omega'}$.
\end{lemma}
\begin{proof}
Let $\pi: M' \rightarrow \widetilde M' $ be the projection and let
$\widetilde h_1, \ldots , \widetilde h_\sigma$ be a basis of the
harmonic 1-forms on $\widetilde M'$. Since $I$ is an orientation
reversing isometry of the orientable surface $M'$, then $I$ leaves
invariant the harmonic $1$-forms $h_i \df \pi^*(\widetilde h_i)$ and
$I^*(\star h_i)=-\star h_i$, where $\star$ denotes the Hodge
operator. Hence, $I^*(\omega_i)=\overline{\omega}_i$, where
$\omega_i \df h_i+ \ri \, \star h_i$. A simple Euler characteristic
calculation shows that $\omega_1 , \ldots, \omega_\sigma$ is a basis
for the holomorphic differentials of $M'$. Let $W=(\omega_1 ,
\ldots, \omega_\sigma)$, then the Abel-Jacobi map $f: M' \rightarrow
\c^\sigma/\Lambda$ satisfies:
\begin{multline} \label{eq:iabel}
f(I(p))= \left[\int_{p_0}^{I(p)} W\right]=\left[ \int_{p_0}^{I(p_0)}
W+ \int_{I(p_0)}^{I(p)} W \right]
=v_0+\left[\int_{p_0}^{p}I^* (W)\right]= \\
v_0+\left[\int_{p_0}^{p} \overline{W} \right]=v_0+c \circ f (p),
\end{multline}
where $c$ is the map on $\c^\sigma/\Lambda$ induced by the complex
conjugation in $\c^\sigma$ and $p_0\in M'$ is a base point.

Let $U \subseteq M'$ be an open region and let $\Div(U)$ denote the
set of divisors in $M'$ whose support is contained in $U$. Then the
map $f$ can be extend linearly to $\Div(U)$ as follows:
$$f\left(\sum_{j=1}^k n_j \cdot p_j \right)=\sum_{j=1}^k n_j\cdot f(p_j).$$
\begin{assertion} \label{as:div}
Let $ \Div_{\sigma-1}(U)$ denote the subset of divisors in $
\Div(U)$ of degree $\sigma-1$. Then $f \colon \Div_{\sigma-1}(U)
\rightarrow \c^\sigma/\Lambda$ is onto.
\end{assertion}
Let $n$ in $\n$ and consider $S_n$ the group of permutations of $(1,
\ldots, n)$. $S_n$ acts on the cartesian product $(M')^n$; the
quotient $S^n(M')$ is called the $n^{\rm th}$ {\em symmetric power
of $M'$}. $S^n(M')$ is a complex manifold of dimension $n$ whose
points can be identified with divisors of the form $D=\sum_{j=1}^n
P_j$. It is well-known \cite[Chap. 15]{nara} that the set of $D \in
S^\sigma ,$ such that the rank at $D$ of the differential of $f
\colon S^\sigma (M') \rightarrow \c^\sigma/\Lambda$ is maximal, $=
\sigma$, is open and dense in $S^\sigma(M')$. In particular,
$f(S^\sigma(U))$ contains an open subset of $\c^\sigma/\Lambda$. So,
if we consider
$$f:S^{n \, \sigma-1}(U) \times S^{(n-1) \sigma}(U) \rightarrow \c^\sigma/ \Lambda$$
$$f(D,E)= f(D)-f(E),$$
then the image $f \left(S^{n \, \sigma-1}(U) \times S^{(n-1)
\sigma}(U)\right) \subseteq f(\Div_{\sigma-1}(U))$ contains an open
subset whose diameter diverges, in terms of $n$. This completes the
proof of Assertion~\ref{as:div}. \vspace{3mm}

Consider $\omega$ a nonzero holomorphic $1$-form satisfying
$I^*\omega=\overline{\omega}$, then the divisor of $\omega$ has this
form $(\omega)= \sum_{j=1}^{\sigma-1} p_j+\sum_{j=1}^{\sigma-1}
I(p_j).$ If we label $\mathcal{K}=\sum_{j=1}^{\sigma-1} f(p_j)$,
then \eqref{eq:iabel} implies that $f((\omega))=2
\Re(\mathcal{K})+(\sigma-1) \, v_0,$ where $\Re$ is the map induced
by the real projection $\mbox{\bf Re} \colon \c^\sigma \rightarrow
\r^\sigma.$ If we consider one of the disks $\d_i$ in the complement
of $\overline{M(\cj_0)}$, then Assertion~\ref{as:div} gives the
existence of $D\in \Div(\d_i)$ so that $\deg(D)=\sigma-1$ and
$f(D)=\mathcal{K}$. So, one has that $\deg(D+I(D))=2 \sigma-2$ and
$$f(D+I(D))=\mathcal{K}+c(\mathcal{K})+\deg(D) \, v_0=2
\Re(\mathcal{K})+(\sigma-1) \, v_0=f((\omega)).$$ Abel's theorem
gives the existence of a meromorphic function $h$ on $M'$ such that
$(h)=(w)-D-I(D)$. In other words, the meromorphic $1$-form
$\tau\df \omega/h$ satisfies:
$$ \left( \tau \right)= \left( \overline{I^*(\tau)} \right)=D+I(D).$$
Therefore, $\tau= a \,  \overline{I^*(\tau)},$ for some complex
constant $a \in \c^*.$ Since $I$ is an involution, then we deduce
that $|a|=1.$

If $a=-1$, then  $\omega' \df \ri \tau $ is  the $1$-form that we
are looking for. If not, we define $\omega' \df
\frac{1+\overline{a}}{2} \tau $ and it satisfies the assertions of
this lemma.
\end{proof}

\subsubsection{López-Ros parameters adapted to nonorientable minimal surfaces} \label{subsubsec:h}
In order to obtain that the examples constructed in \cite{density}
were proper, we used  special types of functions with simple poles
at some points near the boundary of the surface and which were
approximated by $1$ in almost the entire surface. To do the same
thing in the nonorientable case, we need to modify the proof of
Lemma 2 in \cite{density} according to the following explanation.

The holomorphic function $\zeta_{i,k}:M(\cj_0) -\{p_i^k\}
\longrightarrow \c$ having a simple pole at $p_i^k$ \cite[subsection
4.1.1, p. 14]{density} must be replaced by a holomorphic function on
$M(\cj_0) -\{p_i^k\}$ having a simple pole at $p_i^k$ and a zero
(not necessarily simple) at $I(p_i^k)$. The existence of such a
function is guaranteed by Noether's gap theorem (see
\cite{farkas-kra}.)

Now, for $\Theta=\left(\lambda_0, \lambda_1,
\ldots,\lambda_{2(\sigma+\E)}\right) \in \r^{2(\sigma+\E)+1}$, we
consider the function $h^\Theta$ (compare with \cite[subsection
4.1.1, equation (3.10)]{density}):
\begin{equation}
h^\Theta= \frac{\lambda_0\, \theta_i^k \, \zeta_{i,j}+
\exp\left(\sum_{j=1}^{2(\sigma+\E)} \lambda_j \, \varphi_j \right)}
{\,\overline{\lambda_0 \, \theta_i^k \, (\zeta_{i,j} \circ I)}+
\exp\left(-\sum_{j=1}^{2 (\sigma+\E)} \lambda_j \, \varphi_j
\right)}.
\end{equation}
Then, the function $h^\Theta$ in subsection 4.1.1 of \cite{density}
must be replaced by this new one and then all the arguments work in
the same way. The reason for changing $h^\Theta$ is because we need
that
$$h^\Theta \circ I= \frac 1{ \, \overline{h^\Theta}\,},$$ in order
to use this function as a López-Ros parameter for nonorientable
minimal surfaces.

This concludes our discussion on how to adapt the proof of
Theorem~\ref{th:first} to the nonorientable case, which completes
the proof of Theorem~\ref{th:limit}.

\subsection{A nonexistence theorem for nonorientable minimal
surfaces properly immersed in smooth bounded domains}

In this section, we describe a topological obstruction to the
existence of certain proper immersions of open nonorientable
surfaces into a given smooth bounded domain. For this description we
need the following definition.

\begin{definition} Let $\cd$ be a smooth bounded domain. We say that a
proper immersion $f:M \to \cd$ of an open surface $M$ is properly
isotopic to a properly embedded surface in $\cd$ if there exists a
proper continuous map $F:M \times [0,1] \to \cd$ such that for each
$t\in [0,1]$, $F_t=F|_{M \times \{t\}}$ is  a proper immersion into
$\cd$, $F_0$ corresponds to $f$ and $F_1$ is a proper embedding.
\end{definition}

\begin{theorem} \label{th:t0} Suppose $\cd$ is a smooth bounded
domain in $\rth$ with boundary  being a possibly disconnected
surface of genus $g$ and $M$ is a properly immersed surface in
$\cd$. If $M$ is properly isotopic to a properly embedded surface in
$\cd$, then $M$ has at most $g$ nonorientable ends\footnote{An end
of a surface $M$ is said to be {\it nonorientable} if every proper
subdomain with compact boundary which represents the end is
nonorientable.}.
\end{theorem}

\begin{proof} Since $M$ is properly isotopic to an embedded
surface $M'$ in $\cd$, then $M$ is homeomorphic to $M'$. In
particular, the number of nonorientable ends of $M'$ and $M$ is the
same. Hence, it suffices to prove the theorem in the special case
that $M$ is properly embedded, a property that we now assume holds.

Arguing by contradiction, suppose $M$ has at least $g+1$
nonorientable ends $e_1, e_2, \ldots e_{g+1}$.  Since $\cd$ is
smooth, then for some small $\varepsilon>0$,
$\overline{\cd(\varepsilon)}=\{ x \in \overline{\cd} \mid
\dist_{\rth}(x,\partial \cd )\leq \varepsilon\}$ is a smooth domain
which is diffeomorphic to $\partial \cd \times [0,1]$, where
$\partial \cd$ is a smooth compact surface of genus $g$. For some
$\varepsilon$ sufficiently small, $\overline{\cd(\varepsilon)} \cap
M$ contains a collection $\{E_1, E_2, \ldots, E_{g+1}\}$ of pairwise
disjoint, proper subdomains of $M$ with compact boundary and such
that $E_i$ represents the end $e_i$ for $i\in \{1,2, \ldots, g+1
\}$. In this case, after reindexing, we may assume that there is a
component $\partial$ of $\partial {\mathcal D}$ of genus $k$ such
that the limit sets $L(E_1), \ldots, L(E_{k+1})$ are contained in
$\partial$.

For some small positive $ \de$ with $\de<\ve$, the surfaces
$\partial_{\ve}, \partial_{\de}$ in ${\mathcal D}$ parallel to
$\partial$ of distance $\ve, \de, $ respectively, are embedded and
the closed region $R(\ve,\de)\subset {\mathcal D}$ bounded by
$\partial_{\ve}\cup \partial_{\de}$ is topologically $\partial
\times [0,1]$. Since each $E_i$ is nonorientable, for $\de$
sufficiently small, $R(\ve, \de)\cap E_i$ contains a connected,
smooth, compact nonorientable domain $F_j$ with $\partial F_j\subset
\partial R(\ve, \de)=\partial_{\ve}\cup \partial_{\de}$ for
$j\in\{1, 2, \ldots, k+1\}. $

Since for each $j\in\{1, 2,\ldots, k+1\}$, $F_j$ is nonorientable
and $R(\ve,\de)$ is orientable, there is a simple closed curve
$\g_j\subset F_j$ such that $\overline{F}_j\cap \overline{\g}_j=1\in
H_0 (R(\ve,\de), \Z_2),$ where $\overline{\g}_j \cap \overline{F}_j$
is the homological intersection number mod 2 of $\g_j$ and $F_j$
relative to $\partial R(\ve, \de)$. Since the domains $F_1, \ldots,
F_{k+1}$ are pairwise disjoint, we conclude that $\overline{F}_i\cap
\overline{\g}_j=\de_{i,j}$ for $i,j \in \{1, 2,\ldots, k+1\}$.

Now let $\a_j$ be a closed curve in $\partial_{\ve}$ which is
homologous in $R(\ve,\de)$ to $\g_j$.  Since $\overline{F}_i\cap
\overline{\g}_j=\de_{i,j},$  then $\overline{\partial F_i \cap
\partial_{\ve}} \cap \overline{\a}_j = \de_{i,j}$, where we consider
${\partial F_i \cap
\partial_{\ve}}$ to represent an element in $H_1(\partial_{\ve},
\Z_2\}$. In particular, the collection of {\bf pairwise disjoint}
simple closed curves that make up $\bigcup_{i=1}^{k+1} \partial F_i
\cap
\partial_{\ve}$ represent at least $k+1$ independent homology
classes in $\HH_1(\partial_{\ve}, \Z_2)$, which is impossible since
$\partial_{\ve}$ is a compact orientable surface of genus $k$. This
contradiction completes the proof of the theorem.
\end{proof}
The following corollary is an immediate consequence of
Theorem~\ref{th:t0}.

\begin{corollary} If $M$ is an open surface with an infinite number of nonorientable ends,
then there does not exist a proper immersion of $M$ into any smooth
bounded domain, such that the immersion is properly isotopic to a
properly embedded surface in the domain.
\end{corollary}

\subsection{The description of the universal domains of Conjecture~\ref{conj:mmn}
}\label{sec:6.3}
\begin{figure}[htbp]
    \begin{center}
        \includegraphics[width=0.65\textwidth]{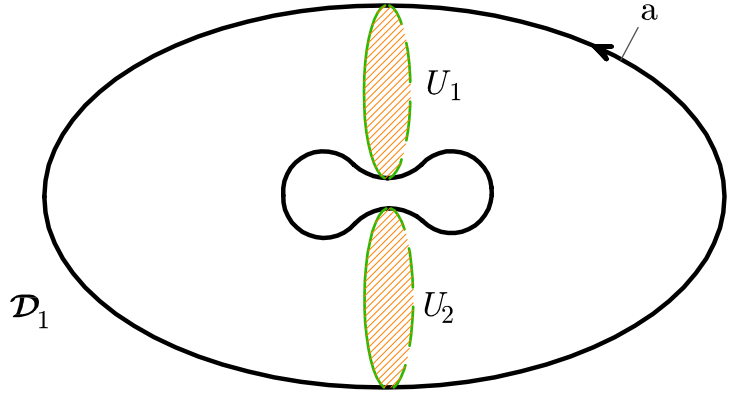}
    \end{center}
   \caption{The domain $\cd_1$, the curve $\mathbf a$ and the disks $U_1$ and $U_2$} \label{fig:basico}
\end{figure}

The main goal of this section is to describe bounded domains of
$\r^3$ which are candidates for solving parts (2) and (3) of the
embedded Calabi-Yau conjecture. From the previous theorem, we know
that some restrictions are necessary in order to properly embed a
nonorientable surface in a smooth bounded domain. That condition is
that the number $n$ of nonorientable ends can not be greater than
the genus of the boundary of the domain. We will actually construct
a sequence of domains $\{\cd_n\}_{n \in \n}$ which are solid
$n$-holed donuts and which contain certain properly embedded
nonorientable minimal surfaces. We conjecture that: \ben
\item If $M$ is a nonorientable open surface with no nonorientable
ends, then it can be properly minimally embedded in $\cd_1$ with a
complete metric.
\item If $n \geq 1$ and  $M$ has $n$ nonorientable ends, then
it can be properly and minimally embedded in $\cd_n$ with a complete
metric. \een

\begin{figure}
    \begin{center}
        \includegraphics[width=0.75\textwidth]{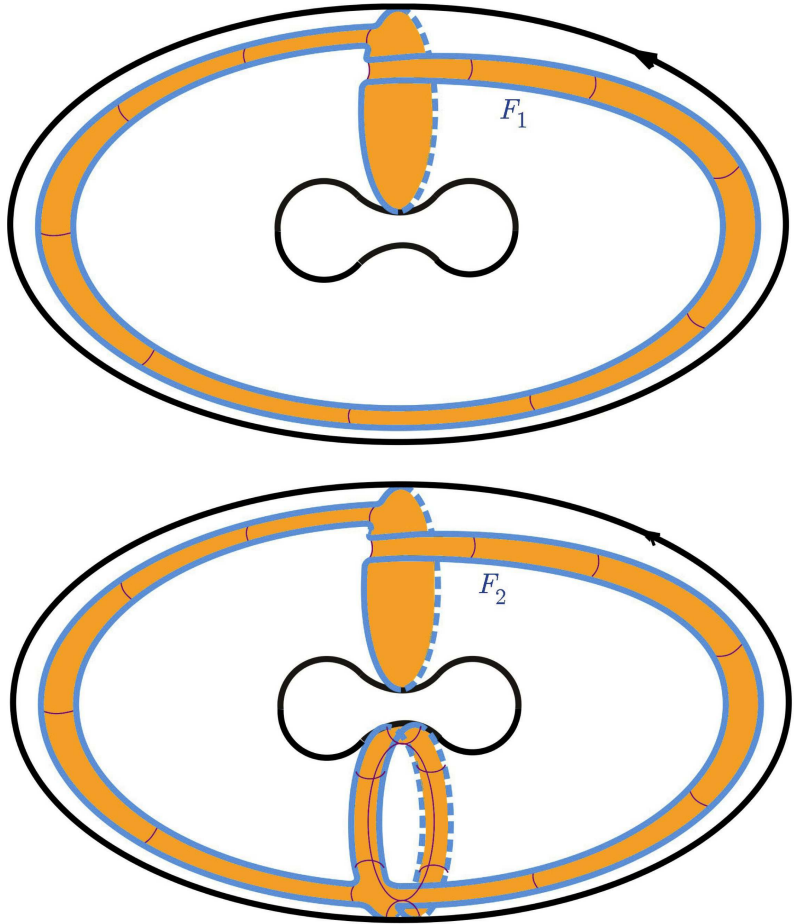}
               \end{center}
   \caption{The minimal surface $F_1$ has the topology of a Möbius strip
    and $F_2$ is topologically a  minimal Klein bottle minus a disk.} \label{fig:fig-4}
\end{figure}

\begin{example} \label{example1}
Consider  a smooth compact solid torus $\overline{\cd_1}$ satisfying
the following properties (see Figure~\ref{fig:basico}): \ben
\item $\overline{\cd_1}$ is invariant under reflections in
the coordinate planes $P_{xy}$, $P_{xz}$ and $P_{yz}.$
\item The intersection of $P_{yz}$ with $\overline{\cd_1}$
consists of two compact convex disks, $U_1$ and $U_2$.
\item The intersection of $P_{xy}$ with $\partial{\cd_1}$
consists of two curves, and the exterior one $\mathbf a$ is  convex.
\item There exists  an open neighborhood $N$ of
$\partial U_1\cup \mathbf{a} \cup \partial U_2 $ in $\partial \cd_1$ with $\kappa_1(N)>1.$

\een
\end{example}
\begin{figure}[htbp]
    \begin{center}
        \includegraphics[width=\textwidth]{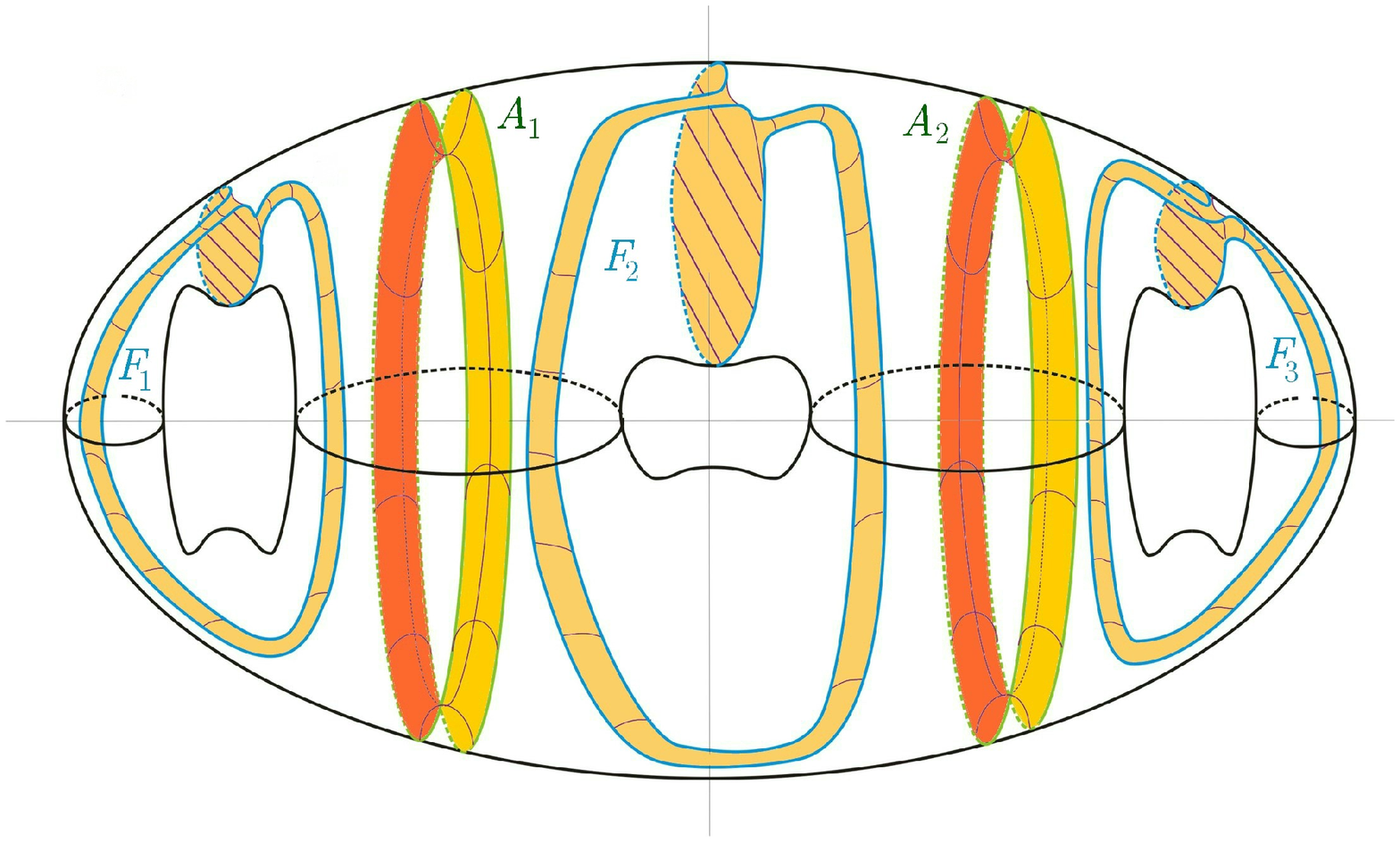}
    \end{center}
    \caption{The domain $\cd_3$}
    \label{fig:fig-1}
\end{figure}

\begin{example}
For $n>1$,  consider now a smooth compact solid $n$-holed torus
$\overline{\cd_n}$ satisfying the following properties (see Figure
~\ref{fig:fig-1} for the case of $\overline{\cd_3}$): \ben
\item $\overline{\cd_n}$ is invariant under reflections
in the coordinate planes $P_{xy}$, $P_{xz}$ and $P_{yz}.$
\item For each integer $k$ in $[-n+1,n-1]$, one of the
components in the intersection of the plane $P_k=\{x=k\}$ and
$\overline{\cd_n}$ is  a compact convex disk, $U_k$ with positive
$y$-coordinate.
\item The intersection of $P_{xy}$ with $\partial{\cd_n}$
consists of $n+1$ curves, and the exterior one $\mathbf{a}$ is  convex.
\item There exists  an open neighborhood $N$ of
$\displaystyle \mathbf{a} \cup \left(\bigcup_{k=-n+1}^{n-1}\partial
U_k \right)$ in $\partial \cd_n$ with $\kappa_1(N) \geq \ve_n >0.$

\een
\end{example}
Finally, we described the domain $\cd_\infty$.

\begin{example} \label{ex:universal}
We consider an infinitely many holed solid donut $\cd_\infty$ with a
single nonsmooth point $p_\infty$ on its boundary which is
accumulation point of the holes of $\cd_{\infty}$. This domain
satisfies the following properties (see Figure~\ref{fig:fig-7}):
\ben
\item The domain $\cd_\infty$ is contained in the slab
$\{ 0 \leq x \leq 1\}$ and $p_\infty=(1,0,0).$
\item $\overline{\cd_\infty}$ is invariant under reflections in
the coordinate planes $P_{xy}$ and $P_{xz}.$
\item There exists  positive real numbers
$r_n$, $s_n$, $n \in \n$, such that:
\ben
\item $r_1<s_1<r_2<s_2<r_3< \cdots <r_n<s_n<r_{n+1}< \cdots$ and $\lim_n r_n=1,$
\item the planes $\{x=r_n\}$ intersect $\cd_\infty$ in
two convex disks, one of them contained in the half space $\{y>0\}$
that we call $U(r_n)$,
\item  the planes $\{x=s_n\}$ intersect $\cd_\infty$
in one convex disk, which we call $V(s_n)$,
\een
\item The intersection of $P_{xy}$ with $\partial{\cd_\infty}$
contains a unique exterior curve $\mathbf{a}$ which  is  convex and smooth.

\item There exists  an open neighborhood $N$ of
$\displaystyle \mathbf{a} \cup \left(\bigcup_{k=1}^{\infty}\partial
U(r_k) \cup \partial V(s_k) \right)$ in $\partial
\cd_\infty-\{p_\infty\}$ with $\kappa_1(N) \geq \ve_\infty>0,$ for
some positive $\ve_\infty$.

\een
\end{example}
Using the bridge principle, the classification of noncompact
surfaces and a suitable choice of a compact exhaustion, we next
prove the following proposition.

\begin{proposition} \label{th:no} For every $n\in \n$, the smooth
domain ${\mathcal D_n}$ satisfies:
\begin{enumerate}
\item For any nonorientable open surface $M$ with no
nonorientable ends, there exists a proper, stable, minimal
noncomplete embedding $f\colon M\to \cd_1$.
\item For any open surface $M$ with $n$
nonorientable ends, there exists a proper, stable, minimal
noncomplete embedding $f\colon M\to \cd_n$.
\end{enumerate}
Furthermore, the embedding $f$ satisfies that the limit sets of
distinct ends of $f(M)$ are disjoint.
\end{proposition}
\begin{proof}
We are going to divide the proof into the case where $M$ has
orientable ends and the case where $M$ has $n$ nonorientable ends.

\noindent {\bf Case 1. $M$ is nonorientable and it has orientable
ends.} By the classification of compact nonorientable surfaces,
there exists a compact exhaustion of $M$, $\cm=\{M_k\mid k\in \n\}$,
such that:
\begin{itemize}
\item $M_1$ is either a Möbius strip or a Klein
bottle with a disk removed, and $M-M_1$ is orientable.
\item Consider the surface $M'$ formed by attaching
a disk $D$ along the
boundary of $M-M_1$ and the associated exhaustion $\cm'=\{M'_1=D, M_k'=M_k\mid k \geq 2\}.$
Then the new exhaustion $\cm'$ is a simple exhaustion of $M'$.
\end{itemize}

Recall from the description in Example~\ref{example1} that $N$ is an
open neighborhood of $\partial U_1 \cup \mathbf{a} \cup \partial
U_2$. Consider a simple arc $\Gamma$ in $N$ with distinct end points
on $\partial U_1$ and which is almost parallel to $\mathbf{a}$. Let
$F_1$ be the compact embedded minimal Möbius strip obtained by
adding a thin bridge to $U_1$ along $\Gamma$ as described in
Figure~\ref{fig:fig-4}. Notice that we can guarantee that $\partial
F_1 \subset N$ by choosing the bridge thin enough. Let $F_2$ be the
embedded compact Klein bottle minus a disk obtained by adding a thin
bridge along $\partial U_2$ to the surface $F_1$ in such a way that
$\partial F_2 \subset N$ as in Figure~\ref{fig:fig-4}.

We now describe how to construct the desired proper minimal
immersion. If $M_1$ is a Möbius strip, then we choose $\Sigma_1$ to
be $F_1$. Since $M-M_1$ is an orientable surface with a ``simple
exhaustion'' and  $\kappa_1(N)>1$, then we can follow the proof of
Case 2 in Theorem~\ref{th:first} in order to construct a proper
minimal embedding $f:M \to \cd_1$ such that the limit set of
different ends of $M$ are disjoint. Of course, this construction is
now much easier since we do not have to deal with the density
theorem; one just  uses the bridge principle to construct compact
embedded minimal surfaces $\Sigma_n \subset \cd_n$. If $M_1$ is a
Klein bottle with a disk removed, then we take $\Sigma_1=F_2$ and
repeat the same argument to construct the desired immersion.

\noindent {\bf Case 2. $M$ is nonorientable and it has $\mathbf n$ nonorientable ends.}

Using again the classification of compact nonorientable surfaces and
arguments similar to those in the proof of Lemma~\ref{lem:simple},
there exists a compact exhaustion of $M$, $\cm=\{M_k\mid k\in \n\}$,
such that:
\begin{itemize}
\item $M_1$ is the compact nonorientable surface with $n$ boundary
components and Euler characteristic $\chi(M_1)=-2 n+1$.
\item Every boundary curve of each $M_k$ separates $M$ into two
components.
\item For each $k \in \n$, $M_{k+1}-\Int(M_k)$ contains exactly
one nonannular component $\Delta_{k+1}$ which is either a Möbius
strip minus a disk, a pair of pants, or an  annulus with a handle.
\item If $\Delta_{k+1}$ is an annulus with a handle, then the
component of $M-\Int(M_k)$ which contains $\Delta_{k+1}$ is orientable.
\item If $\Delta_{k+1}$ is a pair of pants, then at most one
of the two components of $M-\Int(M_{k+1})$ which intersects
$\partial \Delta_{k+1}$ is nonorientable.
\end{itemize}

For the following construction of the domain $\cd_3$, see
Figure~\ref{fig:fig-1}. The planes $P_{-n+2}$, $P_{-n+4}$, $\ldots$,
$P_{n-2}$, separate $\partial \cd_n$ into $n$ open regions that we
call $R_1$, $R_2$, $\ldots$, $R_n$ and which are ordered by their
relative $x$-coordinates. Let $A_1$, $A_2$, $\ldots$, $A_{n-1}$ be
compact stable minimal  annuli in $\cd_n$ with $\partial A_i \subset
\partial \cd_n$, ordered by their relative $x$-coordinates, with boundaries
close and parallel to the boundaries of the regions $R_1$, $R_2$,
$\ldots$, $R_n$, respectively. Let $F_1$, $F_2$, $\ldots$, $F_n$ be
the compact stable minimal Möbius strips with $F_i \subset R_i$,
$i=1, \ldots,n$, constructed by attaching bridges to  the disks
$U_{-n+1}$, $U_{-n+3}$, $\ldots$, $U_{n-1}$ in a manner similar to
the construction of $F_1$ in Case 1. Furthermore, we can assume that
the boundary curves of these annuli and Möbius strips are contained
in the neighborhood $N$. We obtain our surface $\Sigma_1$ by
connecting the annuli and Möbius strips by thin minimal bridges in
$N$ close to the intersection of $P_{xz}$ and $\partial \cd_n$ and
where $z>0$. Finally, we can also assume that $\partial \Sigma_1
\subset N$ and $\Sigma_1$ has $n$ boundary curves $b_i \subset R_i$,
$i=1, \ldots, n$, where we fix an orientation of each boundary
curve.

We now describe how to finish the construction of the desired proper
minimal immersion. In Case 1, the changes in the topology of
$\Sigma_m$, $m \in \n$, occur near  one (prescribed) point in the
boundary of $\Sigma_1. $ In our case, we prescribe $n$ points, $p_i
\in b_i$, $i=1, \ldots,n$, where $p_i$ lies on the boundary of the
bridge used to make $F_i$. The process of adding a pair of pants or
an annulus with a handle to $\Sigma_m$ is the same as in the
orientable case; one attaches a very thin bridge $B$ near a point of
the boundary of $\Sigma_n$ or one attaches $B$ and then a second
bridge $B'$ in the center of $B$ in order to attach an annulus with
a handle (see Figures~\ref{fig:ends-3}.)

The process to add a Möbius strip to $\Sigma_m$ is by attaching a
very thin bridge $B$ along a short oriented simple arc in
$N-\partial \Sigma_m$ with end points on an oriented component
$\gamma \subset \partial \Sigma_m$ and which has the same
intersection numbers with $\gamma$ at each of its end points. For
example, suppose that  $\Delta_2$ is a Möbius strip attached to
$\partial M_1$ along a  boundary component corresponding to  $b_i
\subset \Sigma_1$. In this case, we choose a short arc $\Gamma$
connecting $p_i$ to its opposite point $\widehat p_i$ in the
corresponding bridge used to produce the Möbius strip $F_i$ (see
Figure~\ref{fig:moebius}). Note that the intersection number of
$\Gamma$ with $b_i$ at $p_i$ is opposite to the intersection number
at $\widehat p_i$. In this case we add a bridge $B_1$ to $\Sigma_1$
along $\Gamma$ like in Figure~\ref{fig:moebius} to make $\Sigma_2$.
Since the component of $M-M_1$ containing $\Delta_2$ has exactly one
nonorientable end, then there exists a smallest $k>2$ such that
$\Delta_k$ is a Möbius strip minus a disk  contained in this
component. So, in the construction of $\Sigma_k$, we will again
attach a bridge, this time inside $B_1$ (see
Figure~\ref{fig:moebius}.)
\begin{figure}[htbp]
    \begin{center}
        \includegraphics[width=0.56\textwidth]{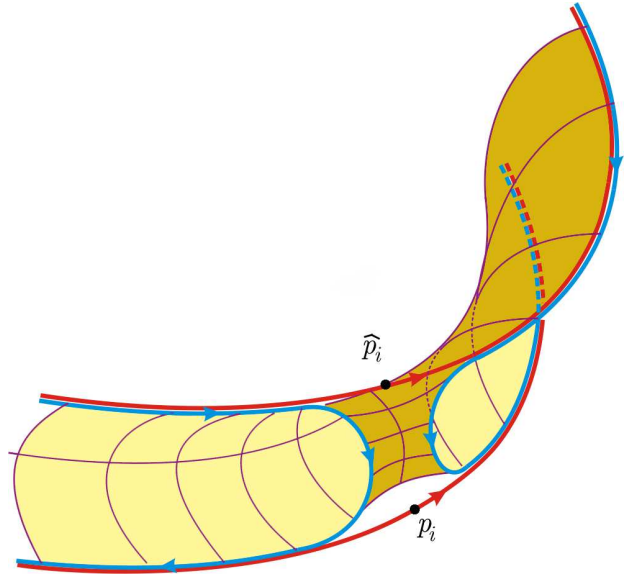}
    \end{center}
   \caption{We choose a short arc $\Gamma$
connecting $p_i$ to its opposite point $\widehat p_i$ in the
corresponding bridge used to produce the Möbius strip $F_i$. Note that the intersection number of
$\Gamma$ with $b_i$ at $p_i$ is opposite to the intersection number
at $\widehat p_i$. In this case we add a bridge $B_1$ to $\Sigma_1$
along $\Gamma$.} \label{fig:moebius}
\end{figure}
Combining all the arguments described in the last two paragraphs, we
obtain a limit surface $\Sigma$ contained in $\cd_n$ and satisfying
all of the statements of the proposition except stability. By
choosing the bridges in the construction of $\Sigma$ sufficiently
thin, then $\Sigma$ is also stable.
\end{proof}

\begin{proposition}\label{prop:infinity}
Every open surface $M$ admits a proper stable minimal embedding in $\cd_\infty.$

\end{proposition}
\begin{proof}

Without loss of generality we can assume that $M$ is not
simply-connected, since $U (r_1)$ is simply-connected and properly
embedded in $\cd_\infty$.

Let $X_1 = \cd_\infty \cap \{x <s_1 \}$ and for $n>1$, define $X_n=
\cd_\infty \cap \{s_{3 n}<x <s_{3n+1} \}$. For each $n\in N$, define
$Y_n= \cd_\infty \cap \{s_{3 n-2}<x <s_{3n-1} \}$, and $Z_n=
\cd_\infty \cap \{s_{3 n-1}<x <s_{3n} \}$. In each region $X_n$ we
construct a compact stable embedded minimal Möbius strip $F_n$ by
attaching a thin bridge to the disk $U(r_{3 n-2})$ like in
Proposition~\ref{th:no}. Similarly, in each $Y_n$ let $A_n$ be a
compact stable embedded minimal annulus close to the boundary of
$U(r_{3 n-1})$. Finally, in each region $Z_n$ we construct a stable
compact embedded minimal disk with a handle $H_n$ by attaching a
bridge to a stable compact minimal annulus near the boundary of
$U(r_{3 n})$. Note that the collection $\{X_n,Y_n,Z_n\}_{n \in \n} $
is a pairwise disjoint family of compact domains whose union is a
properly embedded surface with boundary in $\cd_\infty -
\{p_\infty\}$, where $p_{\infty}=(1,0,0)$. We can assume that the
curve $\mathbf{a}$ intersects all of these compact stable surfaces,
$F_n$, $A_n$, $H_n$, $n \in \n$ and $\bigcup_{n=1}^{\infty}(\partial
F_n\cup
\partial A_n\cup
\partial H_n)\subset N$ (see Figure~\ref{fig:fig-7}).

The case where $M$ has finite topology is easily obtained by
connecting a finite number of the components $F_n$, $A_n$ and $H_n$
by bridges. Hence, from now on we assume that $M$ has infinite
topology.

Following similar ideas to those in the proof of Case 2 in the
previous proposition, we can choose a compact exhaustion of $M$ such
that:
\begin{itemize}
\item $M_1$ is a Möbius strip, an annulus or a disk with a handle.
\item Every boundary curve of each $M_k$ separates $M$ into
two components, one of them containing $M_1$.
\item For each $k \in \n$,  $M_{k+1}-\Int(M_k)$ contains exactly
one nonannular component $\Delta_{k+1}$ which is either a Möbius
strip minus a disk, a pair of pants, or an  annulus with a handle.

\end{itemize}
\begin{figure}[!h]
    \begin{center}
        \includegraphics[width=\textwidth]{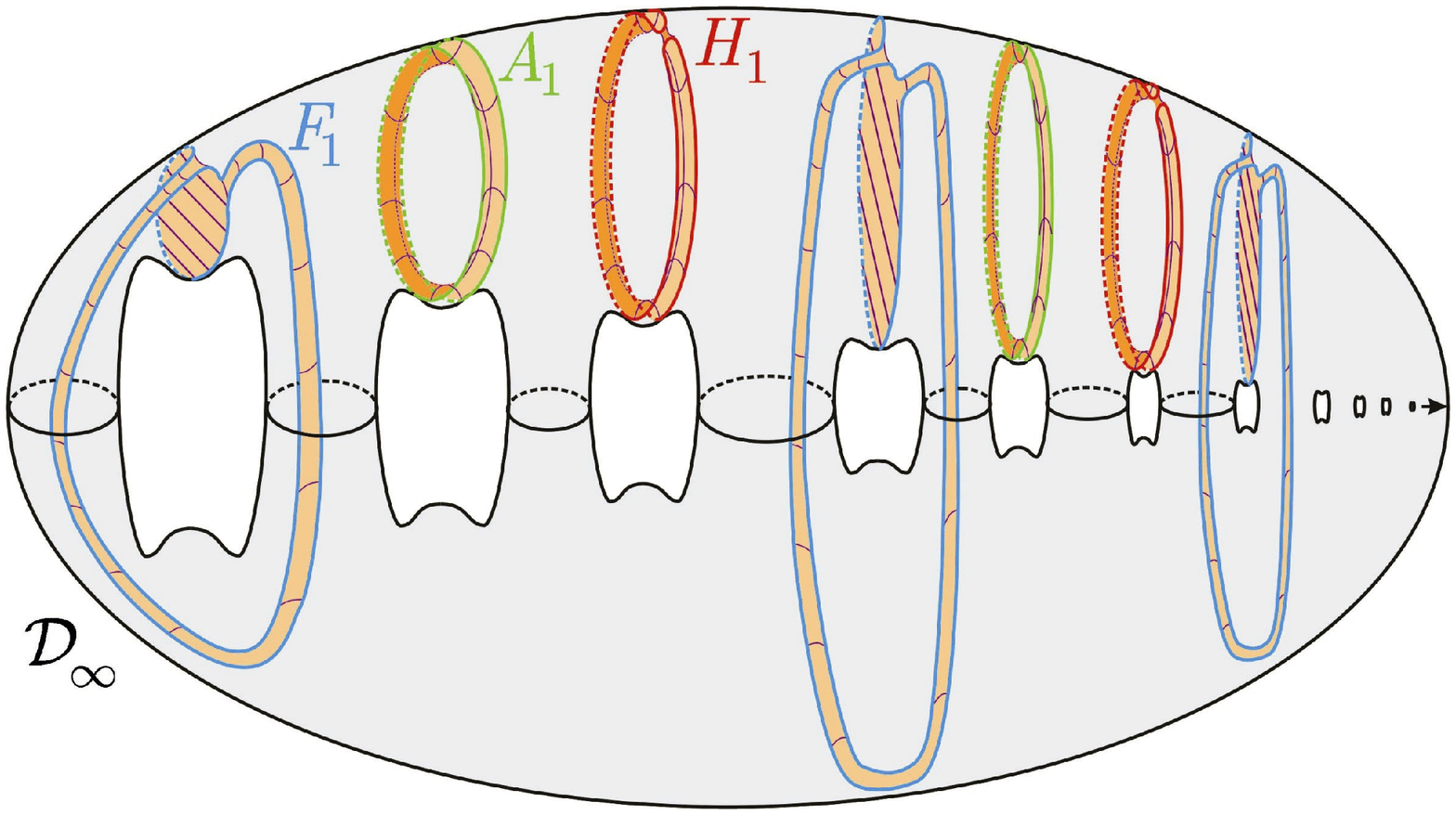}
    \end{center}
   \caption{The domain $\cd_\infty$} \label{fig:fig-7}
\end{figure}

Once again we construct the surface inductively. To do this we only
need to explain how to apply the bridge principle to add a pair of
pants, an annulus with a handle or a Möbius strip to a given
$\Sigma_m$. To guarantee the stability of $\Sigma_m$, we choose
bridges sufficiently narrow.

 Let $\Sigma_1$ be either $F_1$,
$A_1$ or $H_1$ depending on the topology of $M_1$. Then we connect
$\Sigma_1$ to the compact minimal surface $W_2$ in  $\{F_2, A_2,
H_2\}$, which is homeomorphic to $\Delta_2 \subset M_2- \Int (M_1)$
with a disk added to its boundary, by a thin bridge contained in $N$
to make the compact embedded minimal surface $\Sigma_2$. We can do
this connection along an arc that travels from a point in $\partial
\Sigma_1 \cap {\bf a} $ to a point in $\partial W_2 \cap {\bf a}$.

%
%
The surface $\Sigma_m$ is obtained from $\Sigma_{m-1}$ by first
finding a connection curve $\g(m)$ joining a component
$\partial_{m-1}$ of $\partial \Sigma_{m-1}$ to the boundary of one
of the surfaces $W_m \in \{F_m, \, A_m, \, H_m\}$, where $W_m$
depends on the topology of $\Delta_m$. For the construction to work
well, it is helpful that $\g(m)$ be chosen to be contained in a
particular domain $C_{\infty}^m \subset N$ which is defined
inductively as follows. For $\g(k)$, $1\leq k \leq m-1$, there
exists a small regular neighborhood strip $N(\g(k))\subset
C_{\infty}^k \subset N -\left[\partial \Sigma_{k-1} \cup
\left(\cup_{i= 1}^{k-1} N(\g(i))\right) \right]$, which is a positive
distance from $\partial N \cup \left( \cup _{i= 1}^{k-1} N(\g(i)) \right)$ and 
so that $N(\gamma(k))$ contains 
the normal projection to $\partial \cd_\infty$ of the bridge along $\g(k)$. Then
$C_{\infty}^m$ is the connected component of $N-\left[ \partial
\Sigma_m \cup \left(\cup_{i=1}^m N(\g(i)) \right)\right]$ which contains
$p_{\infty}$ in its closure. Furthermore, each $\g(k)$ can be chosen
so that it intersects each $V(s_i)$ transversely in at most one
point.  In particular, we may assume that $N(\g(k))\cap
x^{-1}([s_i, s_{i+1}])  \subset \partial
\mathcal{D}_{\infty}$ is either empty,  a thin strip which
intersects each of the boundary components of $x^{-1}([s_i,
s_{i+1}]) \cap \partial \mathcal{D}_{\infty}$ in a compact arc or
a thin strip which intersects only one of
the boundary curves of $x^{-1}([s_i, s_{i+1}]) \cap \partial
\mathcal{D}_{\infty}$ and this intersection is a connected arc; the
last case occurs when $\g(k)$ intersects the boundary of
$x^{-1}([s_i, s_{i+1}]) \cap \partial \mathcal{D}_{\infty}$ in a
single point, which happens exactly twice. Let $i(0,k) <i(1,k)$ be
the natural numbers so that $\gamma(k)$ intersects 
$\partial V(s_{i(0,k)})$ and $\partial V(s_{i(1,k)})$ in exactly one
point, respectively.

Given a $n \in \n$, assume that $\Sigma_n$ has been constructed and we will construct
$\Sigma_{n+1}$ satisfying all of the properties mentioned in the
previous  paragraph.  Let $\partial_n  \subset \partial \Sigma_n$ be
the component of $\partial \Sigma_n$ which corresponds to $\partial
\Delta_{n+1} \cap \partial M_n$ and let $p_n$ be a point of
$\partial_n$ with largest $x$-coordinate.  Observe that $x(p_n) \in
[s_{i(0,n+1)-1}, s_{i(0,n+1)}]$.  We next
describe in detail how to construct $\g(n+1)$.

{\bf Case A}: $V(s_{i(0,n+1)})\cap \partial \Sigma_n =\emptyset$. In
this case $\g(n+1)$ can be constructed from a small perturbation of
the union of an arc $\beta_0 $ joining $p_n$ to $V(s_{i(0,n+1)})$, where
$ \beta_0$ is contained in $ C_n^{\infty} \cap x^{-1}([s_{i(0,n+1)-1}, s_{i(0,n+1)}]) ,$
and  an arc $\beta_1 \subset \left( V(s_{i(0,n+1)}) \cup \bf a \right)$
with one  end point in $\partial W_{n+1}$.

{\bf Case B}: $V(s_{i(0,n+1)})\cap \partial \Sigma_n \neq \emptyset$.
First consider an arc $\beta_0 \subset C_n^{\infty} \cap
x^{-1}([s_{i(0,n+1)-1}, s_{i(0,n+1)}]) $ joining $p_n$ to a point $q_1$ of
$V(s_{i(0,n+1)}) \cap \partial N(\g(j_1))$, for some $j_1 < n$. Then, we consider
$\a_1$ the connected component of $\partial N(\g(j_1))$ containing $q_1$
and which is contained in $x^{-1} \left( \left[s_{i(0,n+1)},s_{i(1,j_1)} \right] \right).$
If $V(s_{i(1,j_1)+1}) \cap \partial \Sigma_n=\emptyset$, then there is an arc
 $\sigma_1 \subset C^n_\infty \cap x^{-1}([s_{i(1,j_1)}, s_{i(1,j_1)+1}])$
connecting the end point of $\a_1$ to a point in 
$\partial V(s_{i(1,j_1)+1})\subset C^n_\infty.$ As in Case A we can
choose an arc $\beta_1 \subset \left( V(s_{i(1,j_1)+1}) \cup \bf a \right)$
with one  end point in $\partial W_{n+1}$ so that $\g(n+1)$ is a small
perturbation of $\beta_0 \cup \a_1 \cup \sigma_1 \cup \beta_0$.

If $V(s_{i(1,j_1)+1}) \cap \partial \Sigma_n \neq \emptyset$, then
we consider the  arc $\sigma_1$ in $\partial V(s_{i(1,j_1)})- N(\g(j_1))$ connecting
the end point of $\a_1$ to a point $q_2$ in $\partial N(\g(j_2)) \cap V(s_{i(1,j_1)}$
for some $j_2<n$. In this situation, let $\a_2$ be the connected arc of 
$\partial N(\g(j_2)) \cap x^{-1}([s_{i(1,j_1)}, s_{(1,j_2)}])$ starting at $q_2$. Repeating
this process a finite number of times we arrive to a curve $\g(j_k)$ so that
$V(s_{i(1,j_k)+1}) \cap \partial \Sigma_n = \emptyset$. Then we proceed like
in the previous paragraph. We consider an arc
$\sigma_k \subset C^n_\infty \cap x^{-1}([s_{i(1,j_k)}, s_{i(1,j_k)+1}])$
connecting the end point of the corresponding arc  $\a_k$ to a point in 
$\partial V(s_{i(1,j_k)+1})\subset C^n_\infty.$ Finally, we can
choose an arc $\beta_1 \subset \left( V(s_{i(1,j_k)+1}) \cup \bf a \right)$
with one  end point in $\partial W_{n+1}$ so that $\g(n+1)$ is a small
perturbation of $\beta_0 \cup \a_1 \cup \sigma_1 \cup \a_2 \cup \sigma_2 \cup \cdots \cup \a_k \cup \sigma_k \cup \beta_1.$
 \vspace{.2cm}

It is important to notice that the compact embedded minimal surfaces
$\Sigma_n$, $n \in \n$, satisfy that for any $ r \in (0,1)$ the boundary
of $\Sigma_n$ intersects $\{x \leq r\}$ in  the same set of arcs and
closed curves, for $n$ sufficiently large. So, there is a bound on
the area of $\Sigma_n \cap  \{x \leq r\}$, independent of $n$. 
Since the surfaces $\Sigma_n$ are embedded and
stable, then a subsequence of them converges on compact subsets of
$\overline{\cd_\infty}-\{p_\infty\}$ to a limit minimal surface
$\Sigma $ with boundary and which is properly embedded in
$\overline{\cd_\infty}-\{p_\infty\}$ and so that $\Sigma \cap
\cd_\infty$ has the topology of $M$. By boundary regularity, the
limit surface $\Sigma$ is smooth. Moreover, if we choose our
connecting bridges  sufficiently thin, then we can guarantee that
the limit surface is unique.

\end{proof}

\begin{remark}
If we combine the arguments in the previous proof with the density
theorem (including the nonorientable version) one can show that
every open surface $M$ admits a complete proper minimal immersion in
$\cd_\infty$ which is properly isotopic to the minimal embedding
given in Proposition~\ref{prop:infinity}.  Similarly,
Proposition~\ref{th:no} can be adapted to produce complete proper
minimal immersions of a given nonorientable open surface $M$ with $n
\in \n$ nonorientable ends into $\cd_n$ in such a way that the
immersion is properly isotopic to the minimal embedding provided by
the proposition and such that the limit sets of distinct ends are
disjoint (if $M$ has orientable ends, then the immersion lies in
$\cd_1$). Taking Theorem~\ref{th:t0} into account, this last result
is sharp.
\end{remark}

\end{document}